\documentclass[11pt,a4paper]{amsart}
\usepackage{amsmath,amssymb, amsbsy}
\usepackage{color,psfrag}
\usepackage[dvips]{graphicx}
\numberwithin{equation}{section}
\usepackage{enumerate}
\newcommand{\R}{{\mathbb R}}
\newcommand{\N}{{\mathbb N}}

\renewcommand{\ge }{\geqslant}
\renewcommand{\geq }{\geqslant}

\renewcommand{\leq }{\leqslant}
\def\neweq#1{\begin{equation}\label{#1}}
\def\endeq{\end{equation}}
\def\eq#1{(\ref{#1})}
\newtheorem{theorem}{Theorem}[section]
\newtheorem{proposition}[theorem]{Proposition}
\newtheorem{lemma}[theorem]{Lemma}

\newtheorem{remark}[theorem]{Remark}

\theoremstyle{definition}

\begin{document}

\title[Non-homogeneous partially hinged plates]{Maximizing the ratio of eigenvalues of\\ non-homogeneous partially hinged plates}

\author[Elvise BERCHIO]{Elvise BERCHIO}
\address{\hbox{\parbox{5.7in}{\medskip\noindent{Dipartimento di Scienze Matematiche, \\
Politecnico di Torino,\\ Corso Duca degli Abruzzi 24, 10129 Torino, Italy. \\[3pt]
\em{E-mail address: }{\tt elvise.berchio@polito.it}}}}}
\author[Alessio FALOCCHI]{Alessio FALOCCHI}
\address{\hbox{\parbox{5.7in}{\medskip\noindent{Dipartimento di Scienze Matematiche, \\
Politecnico di Torino,\\ Corso Duca degli Abruzzi 24, 10129 Torino, Italy. \\[3pt]
\em{E-mail address: }{\tt alessio.falocchi@polito.it}}}}}

%\date{\today}



\begin{abstract} 
	We study the spectrum of non-homogeneous partially hinged plates having structural engineering applications. A possible way to prevent instability phenomena is to maximize the ratio between the frequencies of certain oscillating modes with respect to the density function of the plate; we prove existence of optimal densities and we investigate their analytic expression. This analysis suggests where to locate reinforcing material within the plate; some numerical experiments give further information and support the theoretical results.

\end{abstract}

\maketitle

	\section{Introduction}
	%\label{1}	
	In recent years the trend in bridge design is to replace expensive experiments in wind tunnels with numerical tests; hopefully, these tests should be preceded by a suitable mathematical modelling and, possibly, by analytic arguments. In particular, since it is by now well-established that reliable models for suspension bridges should have enough degrees of freedom to display torsional oscillations, it is convenient to model the deck of the bridge by means of a long narrow rectangular thin plate $\Omega\subset \R^2$, hinged at short edges and free on the remaining two, see \cite{fergaz} and problem \eqref{weight} below. \par  When the wind comes up against the deck of the bridge, a form of dynamic instability arises, which appears as uncontrolled vortices and it is usually named \emph{flutter}. The origin of asymmetric vortices generates a forcing lift which launches vertical oscillations of the deck; this phenomenon finds confirmation in wind tunnel tests, see e.g. \cite{larsen}. In particular, a transition between these vertical oscillations to torsional ones may happen which, in some cases, leads to the collapse of the bridge; we refer to \cite[Chapter 1]{bookgaz} for a survey of historical events where this phenomenon occurred, among which the infamous Tacoma Narrows Bridge collapse. Therefore, it becomes extremely important preventing flutter instability to provide a structure strong and safe.
 Rocard \cite{rocard} suggested that for common bridge there exists a threshold of wind velocity $V_c$ at which flutter arises. 
% In this regime the deck has high energy and phenomena of internal resonance may appear, leading to possible transfer of energy between different oscillating modes. Denoting by $V_c$ the critical velocity of the wind, the threshold for flutter and for internal resonance are linked by the kinetic energy
%$$E_c=kV_c^2$$
%for some $k>0$ depending on the air density, e.g. see \cite{herrmann}. 
The computation of $V_c$ is not an easy task, since it depends on the wind and on the geometric features of the deck; a possible way is to determine it experimentally. On the other hand, in engineering literature there exist some closed formulas for $V_c$; even if the debate on these formulas is still open, it seems to be accordance in thinking that the critical velocity depends on the frequencies or, equivalently, on the eigenvalues of the normal modes of the deck, see \cite{como,Irvine,rocard}. More precisely, since $V_c$ represents the critical threshold at which an energy transfer occurs between the $j$-th and the $i$-th mode of oscillation, most of the authors propose $V_c$  directly proportional to the difference between the square of the corresponding eigenvalues $\lambda_i> \lambda_j$, i.e.
$$V_c \propto (\lambda_i^2-\lambda_j^2)\,.$$
It follows that a way to increase the critical velocity $V_c$, and in turn to prevent instability, is by increasing the distance between $\lambda_i^2$ and $\lambda_j^2$; this purpose is achievable moving the ratio $(\lambda_i/\lambda_j)^2$ away as much as possible from $1$. A theoretical explanation of this fact was given in \cite{bgz}, within the classical stability theory of Mathieu equations, by relating large ratios of eigenvalues to the situation in which the instability resonant tongues of the Mathieu diagram become very thin. \par
Coming back to the plate model of the bridge, in order to prevent dynamical instability, different strategies to optimize the design of the plate have been proposed in literature; for instance, one may modify its shape, see \cite{bebuga2}, or rearrange the materials composing it, see \cite{bebugazu,befafega}. Within the present research, we exploit the latter approach to maximize the ratio of selected eigenvalues of a partially hinged non-homogeneous plate. More precisely, by rescaling, we assume that the plate has length $\pi$ and width $2\ell$ with $2\ell\ll\pi$ so that
	$$
	\Omega=(0,\pi)\times(-\ell,\ell)\subset\R^2\,;
	$$
then we characterize the non-homogeneity of the plate by a density function $p=p(x,y)$ and we consider the weighted eigenvalues problem:
	\begin{equation}\label{weight}
	\begin{cases}
	\Delta^2 u=\lambda\, p(x,y) u & \qquad \text{in } \Omega \\
	u(0,y)=u_{xx}(0,y)=u(\pi,y)=u_{xx}(\pi,y)=0 & \qquad \text{for } y\in (-\ell,\ell)\\
	u_{yy}(x,\pm\ell)+\sigma
	u_{xx}(x,\pm\ell)=u_{yyy}(x,\pm\ell)+(2-\sigma)u_{xxy}(x,\pm\ell)=0
	& \qquad \text{for } x\in (0,\pi)\, .
	\end{cases}
	\end{equation}
	The boundary conditions on short edges are of Navier type, see \cite{navier}, and model the situation in which the deck of the bridge is hinged on $\{0,\pi\}\times(-\ell,\ell)$. Instead, the boundary conditions on large edges are of Neumann type, see \cite{copro,provenzano}, they model the fact that the deck is free to move vertically and involve the Poisson ratio $\sigma$ which, for most of materials, satisfies $\sigma\in(0,1/2)$. Finally, we focus on densities $p$ satisfying some natural constraints, i.e. for $\alpha,\beta \in (0,+\infty)$ with $\alpha<\beta$ fixed, we assume that $p$ belongs to the following class of weights 
\begin{equation} \label{eq:famiglia}
P_{\alpha, \beta}:=\left\{p\in L^\infty(\Omega):\alpha\leq p\leq\beta\,,\quad  p(x,y)=p(x,-y)  \ \text{a.e. in } \Omega\\ \text{ and }\\ \int_{\Omega}p\,dxdy=|\Omega| \, \right\} \, .
\end{equation}
The integral condition in \eqref{eq:famiglia} represents the preservation of the total mass of the plate, while the symmetry requirement on $p$ means that we focus on designs which are symmetric with respect to the mid-line of the roadway. From a mathematical point of view, the symmetry of $p$ produces two classes of eigenfunctions of \eq{weight}, respectively, even or odd in the $y$-variable, that we named \emph{longitudinal} and \emph{torsional} modes. In order to prevent the energy transfer from longitudinal to torsional modes, one may study the effect of the weight $p$ on the ratio $\nu(p)/\mu(p)$, where $\nu$ and $\mu$ are two selected eigenvalues corresponding, respectively, to a torsional and a longitudinal mode. Since the final goal is to find the best rearrangement of materials in $\Omega$ which maximizes this ratio, we study, either from a theoretical and a numerical point of view, the optimization problem:
\begin{equation}\label{opt_intro}
\mathcal{R}=\sup_{p\in P_{\alpha,\beta}}\dfrac{\nu(p)}{\mu(p)}.
\end{equation}
We refer to \cite{benguria} for optimization results on the ratio of eigenvalues of second order operators subject to domain perturbations and to \cite{keller} for optimization results, with respect to the weight, in $1$-dimensional domains; see also \cite[Chapter 9]{henrot} and references therein. In particular, in \cite{keller} the author proved that the weight maximizing the considered ratio is of \emph{bang-bang} type, namely a piecewise constant function, symmetric with respect to the middle of the string and getting the minimum value there. Unfortunately, the techniques exploited in \cite{keller}  are closely related to the 1-dimensional nature of the problem and seem not applicable to our situation. Furthermore, here, things are complicated by dealing with a fourth order operator with non standard boundary conditions, for which no general positivity results are known. We refer the interested reader to \cite{befafega} where a partial positivity property result was proved for the operator in \eq{weight}.\par  As a consequence of what remarked, at the current state of art, a complete theoretical solution to problem \eq{opt_intro} is difficult to reach and we proceed by steps. More precisely, we concentrate our efforts in looking for weights increasing $\nu(p)$ or reducing $\mu(p)$, separately. The numerical results we collect in Section \ref{num2} reveal that this apparently not rigorous approach turns out to be effective in increasing the ratio \eq{opt_intro}; indeed, as a matter of fact, weights having strong effect on torsional eigenvalues $\nu(p)$ produce very confined effects on longitudinal eigenvalues $\mu(p)$, and viceversa. In this regard, preliminary results were obtained in \cite{befafega}, where the goal was minimizing the \emph{first} eigenvalue of \eqref{weight}, see Proposition \ref{thm-mu1} below. The focus of the present paper is on \emph{higher} eigenvalues, furthermore we deal with a supremum problem and different methods are required; hence, the optimization issue \eq{opt_intro} deserves to be studied independently.  About the optimization of the \emph{first} weighted eigenvalue of $\Delta^2$ under Dirichlet or Navier boundary conditions, we mention the papers \cite{anedda1},\cite{anedda2},\cite{CV}-\cite{cuccu22}. Concerning \emph{higher} eigenvalues we refer to \cite{lapr} where the authors provide a detailed spectral optimization analysis, upon density variations, of general elliptic operators of arbitrary order subject to several kinds of boundary conditions. In \cite{chen} numerical results were given for the Dirichlet and Navier version of of \eq{weight}; while in \cite{copro} sharp upper bounds for weighted eigenvalues in the Neumann case were provided.\par In order to increase the numerator of \eq{opt_intro}, i.e. the first torsional eigenvalue, we adapt to our situation the approach developed by \cite{cox}, in the second order case, and which was partially extended to the fourth order by \cite{cuccu22}, in order to optimize the first biharmonic eigenvalue under Navier or Dirichlet boundary conditions. The main novelty of the present paper is the exploitation of the precise information we have from \cite{fergaz} on the spectrum of problem \eq{weight} with $p\equiv 1$; this fact allows us to partially overcome the loss of positivity results for \eqref{weight}. Moreover, since we work with a domain $\Omega\subset \mathbb{R}^2$ rectangular, we perform some computations explicitly; in particular, we obtain upper bounds on longitudinal eigenvalues that, suitable combined with some rearrangements arguments inspired by \cite{chanillo} and \cite{CV}, give the analytic expression of weights reducing the denominator in \eq{opt_intro}, see Theorem \ref{thm-muj}. Finally, in Section \ref{num2} we complete our theoretical results with numerical experiments; they provide  weights increasing the ratio \eq{opt_intro} and suggest a  maximizer to \eq{opt_intro}.

\par
The paper is organized as follows. In Section \ref{problem} we introduce some preliminaries and notations and we recall the known results in the case $p\equiv 1$. Section \ref{main} is devoted to the main results of the paper, which we prove in Section \ref{proof}. The theoretical results are complemented with numerical experiments collected in Section \ref{num2}, where we give some practical suggestions about the location of the reinforcements in the plate. Finally, in the Appendix we complete our analysis of problem \eqref{weight} by providing a Weyl-type asymptotic law for the eigenvalues. 
 \par

\section{Preliminaries and notations}\label{problem}
From now onward we fix $\Omega=(0,\pi)\times(-\ell,\ell)\subset\R^2$ with $\ell>0$ and $\sigma\in(0,1/2)$. We denote by $\|\cdot\|_q$ the norm related to the Lebesgue spaces $L^q(\Omega)$ with $1\leq q\leq\infty$ and we shall omit the set $\Omega$ in the notation of the functional spaces, e.g. $V:=V(\Omega)$. The natural functional space where to set problem \eq{weight} is
$$
H^2_*=\big\{u\in H^2: u=0\mathrm{\ on\ }\{0,\pi\}\times(-\ell,\ell)\big\}\,.
$$
Note that the condition $u=0$ has to be meant in a classical sense because $\Omega$ is a planar domain and the energy space $H^2_*$ embeds into continuous functions. 
Furthermore, $H^2_*$ is a Hilbert space when endowed with the scalar product
$$
(u,v)_{H^2_*}:=\int_\Omega \left[\Delta u\Delta v+(1-\sigma)(2u_{xy}v_{xy}-u_{xx}v_{yy}-u_{yy}v_{xx})\right]\, dx \, dy \,
$$
and associated norm
$$
\|u\|_{H^2_*}^2=(u,u)_{H^2_*} \, ,
$$
which is equivalent to the usual norm in $H^2$, see \cite[Lemma 4.1]{fergaz}.  Then, we reformulate problem \eq{weight} in the following weak sense
\begin{equation}
\label{eigenweak}
(u,v)_{H^2_*} =\lambda\int_{\Omega}p(x,y)uv\,dx\,dy \qquad\forall v\in H^2_*,
\end{equation}
where $p$ belongs to the family of weights $P_{\alpha,\beta}$ defined in \eqref{eq:famiglia} with  $\alpha,\beta \in (0,+\infty)$ and $\alpha<\beta$ fixed. 
We underline that condition $p\in P_{\alpha,\beta}$ implies $\alpha\leq 1 \leq \beta$ since $\int_\Omega p\,dx\,dy=|\Omega|$. Moreover, it is not restrictive to assume $\alpha<1<\beta$ when we consider weights that do not coincide a.e. with the constant function $p \equiv 1$. In fact, if we assume $\beta= 1$, it must be $p = 1$ a.e. in $\Omega$, since otherwise we would have
$\int_\Omega p\,dx\,dy<|\Omega|$; similarly, if we put $\alpha=1$. For these reasons, since the aim of our research is to study the effect of a non-constant weight on the eigenvalues of \eqref{weight}, in what follows we will always assume $$0<\alpha<1<\beta\,.$$ \par The bilinear form $(u,v)_{H^2_*}$ is continuous
and coercive and $p\in L^\infty$ is positive a.e. in $\Omega$,
by standard spectral theory of self-adjoint operators we infer

\begin{proposition}\label{general facts}
Let $p\in P_{\alpha,\beta}$. Then all eigenvalues of \eqref{eigenweak} have finite multiplicity and can be represented by means of an increasing and divergent sequence $\lambda_h(p)$ ($h\in \N_+$), where each eigenvalue is repeated according to its multiplicity. Furthermore, the corresponding eigenfunctions form a compete system in $H^2_*$.
\end{proposition}
We refer to \cite[Lemma 2.1]{lapr} for a detailed proof of Proposition \ref{general facts} in a more general setting. On the other hand, it is well-known, see  \cite{courant,henrot}, that the following variational representation of eigenvalues holds for every $h\in \N_+$:
\begin{equation}\label{caract1}
\lambda_h(p)=\inf_{\substack{W_h\subset H^2_*\\\dim W_h=h}}\hspace{2mm}\sup_{\substack{u\in W_h\setminus\{0\}}}\frac{\|u\|^2_{H^2_*}}{\|\sqrt{p}u\|^2_2}.
\end{equation} 
When $h=1$, \eqref{caract1} includes the well known characterization for the first eigenvalue
\begin{equation}\label{first}
\lambda_1(p)=\inf_{\substack{u\in H^2_*\setminus\{0\}}}\frac{\|u\|^2_{H^2_*}}{\|\sqrt{p}u\|^2_2}.
\end{equation} 
If $h\geq 2$, the minimum in \eqref{caract1} is achieved by the space $W_h$ spanned by the $h$-th first eigenfunctions. Assuming that the first $h-1$ eigenfunctions, $u_1, u_2, \dots, u_{h-1}$ are known, one also obtains
\begin{equation}\label{caract2}
\lambda_h(p)=\inf_{\substack{u\in H^2_*\setminus\{0\}\\ (u,u_i)_{H^2_*}=0\, \, \forall \, i=1,...,h-1}}\frac{\|u\|^2_{H^2_*}}{\|\sqrt{p}u\|^2_2}\,.
\end{equation}

When  $p\equiv 1$, we recall that the whole spectrum of \eqref{weight} was determined in \cite{fergaz} (see also \cite{bebuga2}); we collect what known in the following proposition. 
\begin{proposition}\label{eigenvalue} \cite{fergaz}
	 Consider problem \eqref{weight} with $p\equiv 1$. Then:
	 \begin{itemize}
	\item[$(i)$] for any $m\ge1$ integer there exists a unique eigenvalue $\lambda=\Lambda_{m,1}\in((1-\sigma^2)m^4,m^4)$ with corresponding eigenfunctions
$\phi_{m,1}(y)\sin(mx)$ with $\phi_{m,1}(y)$ given in \eq{phi};
	\item[$(ii)$] for any $m\ge1$ and any $k\ge2$ integers there exists a unique eigenvalue $\lambda=\Lambda_{m,k}>m^4$ satisfying
	$\left(m^2+\frac{\pi^2}{\ell^2}\left(k-\frac{3}{2}\right)^2\right)^2<\mu_{m,k}<\left(m^2+\frac{\pi^2}{\ell^2}\left(k-1\right)^2\right)^2$, with corresponding eigenfunctions $\phi_{m,k}(y)\sin(mx)$ with $\phi_{m,k}(y)$ given in \eq{phi};
	\item[$(iii)$] for any $m\ge1$ and any $k\ge2$ integers there exists a unique eigenvalue $\lambda=\Lambda^{m,k}>m^4$ with corresponding eigenfunctions $\psi_{m,k}(y)\sin(mx)$ with $\psi_{m,k}(y)$ given in \eq{phi};
	\item[$(iv)$] for any $m\ge1$ integer, satisfying $\ell m\sqrt 2\, \coth(\ell m\sqrt2 )>\left(\frac{2-\sigma}{\sigma}\right)^2$, there exists a unique
	eigenvalue $\lambda=\Lambda^{m,1}\in(\Lambda_{m,1},m^4)$ with corresponding eigenfunctions $\psi_{m,1}(y)	\sin(mx)$ with $\psi_{m,1}(y)$ given in \eq{phi}.
	\end{itemize}
		Finally, if
	\begin{equation}\label{c0}
	\text{the unique positive solution $s>0$ of: }\tanh(\sqrt{2}s\ell)=\left(\frac{\sigma}{2-\sigma}\right)^2\, \sqrt{2}s\ell\quad  \text{is not an integer,}
	\end{equation}
	then the only eigenvalues are the ones given in $(i)-(iv)$.
\end{proposition}
In the following, we will always assume that \eq{c0} holds. \par
At last, we recall the analytic expression of the functions $\phi_{m,k}(y)$ and $\psi_{m,k}(y)$ of Proposition \ref{eigenvalue}. For $m,k\in \N_+$, we define:
{\small \begin{equation} \label{phi}
\begin{split}
\phi_{m,1}(y)&:=\dfrac{1}{N_{m,1}}\bigg\{\,\frac{\sigma m^2-c_{m,1}^2 }{\cosh (\ell\, \overline c_{m,1})} \cosh (y\, \overline c_{m,1}) +
\frac {\overline c_{m,1}^2-\sigma m^2 }{\cosh (\ell\, c_{m,1})}\   \, \cosh (y\, c_{m,1})   \bigg\}\\
\phi_{m,k}(y)&:=\dfrac{1}{N_{m,k}}\bigg\{\,\frac{\sigma m^2+c_{m,k}^2 }{\cosh (\ell\, \overline c_{m,k})} \cosh (y\, \overline c_{m,k}) +
\frac {\overline c_{m,k}^2-\sigma m^2 }{\cos (\ell\, c_{m,k})}\   \, \cos (y\, c_{m,k})   \bigg\}\\ 
\psi_{m,k}(y)&:=\dfrac{1}{N^{m,k}}\bigg\{\,\frac{\sigma m^2+d_{m,k}^2 }{\sinh (\ell\, \overline d_{m,k})} \sinh (y\, \overline d_{m,k}) +
\frac {\overline d_{m,k}^2-\sigma m^2 }{\sin (\ell\, d_{m,k})}\   \, \sin (y\,d_{m,k})   \bigg\}\\ 
\psi_{m,1}(y)&:=\dfrac{1}{N^{m,1}}\bigg\{\,\frac{\sigma m^2-d_{m,1}^2 }{\sinh (\ell\, \overline d_{m,1})} \sinh (y\,\overline d_{m,1}) +
\frac {\overline d_{m,1}^2-\sigma m^2 }{\sinh (\ell\, d_{m,1})}\   \, \sinh (y\, d_{m,1})   \bigg\} \,,
\end{split}
\end{equation}}
where 
\begin{equation*}
\begin{split}
c_{m,k}:=\sqrt{|(\Lambda_{m,k})^{1/2}-m^2|}\qquad \overline c_{m,k}:=\sqrt{ (\Lambda_{m,k})^{1/2}+m^2}\\d_{m,k}:=\sqrt{|(\Lambda^{m,k})^{1/2}-m^2|}\qquad \overline d_{m,k}:=\sqrt{(\Lambda^{m,k})^{1/2}+m^2}\,,
\end{split}
\end{equation*}
with $\Lambda_{m,k}$ and $\Lambda^{m,k}$ defined in Proposition \ref{eigenvalue}, and $N_{m,k}, N^{m,k}\in \R_+$ are fixed in such a way that $\|\phi_{m,k}(y)\,\sin(mx)\|_2=\|\psi_{m,k}(y)\,\sin(mx)\|_2=1$.
\begin{remark}

	Denote by $\lambda_h(1)$ ($h\in \N_+$) the sequence of eigenvalues of \eqref{weight} with $p\equiv 1$; this sequence can be written explicitly by ordering the eigenvalues given by Proposition \ref{eigenvalue}. Then, for all $p\in P_{\alpha,\beta}$, the characterization \eqref{caract1} readily gives the stability inequality
	$$\frac{\lambda_h(1)}{\beta}\leq \lambda_h(p)\leq \frac{\lambda_h(1)}{\alpha}\,,$$
	for every $h\in\mathbb{N}_+$. In applicative terms, if we choose materials having similar densities, we obtain eigenvalues close to those of the homogeneous plate.
	%This estimate does not provide information on the location of the denser material, but it gives a control on the $h$-th weighted eigenvalue with respect to the parameters $\alpha$ and $\beta$, bounding the density functions. 

\end{remark}
 By Proposition \ref{eigenvalue} we distinguish two classes of eigenfunctions of problem \eqref{weight} with $p\equiv 1$: \par\smallskip\par $\bullet$ $y$-even eigenfunctions $\phi_{m,k}(y)\,\sin(mx)$ corresponding to the eigenvalues $\Lambda_{m,k}$;  \par $\bullet$ $y$-odd eigenfunctions $\psi_{m,k}(y)\,\sin(mx)$ corresponding to the eigenvalues $\Lambda^{m,k}$.\par\smallskip\par  As in \cite{bogamo}, this suggests to introduce the subspaces of $H^2_*$:
 \begin{equation*}\label{subspaces}
 \begin{split}
 &H^2_\mathcal{E}:=\{u\in H^2_*: u(x,-y)=u(x,y)\quad\forall (x,y)\in\Omega\},\\&H^2_\mathcal{O}:=\{u\in H^2_*: u(x,-y)=-u(x,y)\quad\forall (x,y)\in\Omega\},
 \end{split}
 \end{equation*}
 where
 \begin{equation*}
 H^2_\mathcal{E}\perp H^2_\mathcal{O}, \hspace{5mm}H^2_*=H^2_\mathcal{E}\oplus H^2_\mathcal{O}\,.
 \end{equation*}
 By the symmetry assumption on $p\in P_{\alpha, \beta}$ it is readily verified that all linearly independent eigenfunctions of \eq{weight} may be  thought in the class $H^2_\mathcal{E}$ or in the class $H^2_\mathcal{O}$. We call the  eigenfunctions belonging to $H^2_\mathcal{E}$ \emph{longitudinal} modes and those belonging to $H^2_\mathcal{O}$ \emph{torsional} modes.  In what follows we order all eigenvalues of \eqref{weight}, repeated according to multiplicity, into two increasing and divergent sequences: the sequence of the eigenvalues $\mu_j(p)$ ($j\in \N_+$) corresponding to longitudinal eigenfunctions and the sequence of the eigenvalues $\nu_j(p)$ ($j\in \N_+$) corresponding to torsional eigenfunctions. From Proposition \ref{eigenvalue} we infer that the sequences $\mu_j(1)$ and $\nu_j(1)$ can be written explicitly by ordering, respectively, the numbers $\Lambda_{m,k}$ and $\Lambda^{m,k}$. In particular, we have 
  \begin{equation}\label{primi}
\lambda_1(1)=\mu_1(1)=\Lambda_{1,1} < \nu_1(1)=\min\{ \Lambda^{1,1},\Lambda^{1,2} \}\,.
 \end{equation}
For actual bridges, one usually has $\nu_1(1)=\Lambda^{1,2}$, indeed the inequality required in Proposition \ref{eigenvalue}-$iv)$ is not satisfied for $\ell$ small, see Table \ref{eig num} in Section \ref{num2}. We note that, even in the case $p\equiv1$, simplicity of eigenvalues is not know, hence, in principle, the same eigenvalue may correspond either to longitudinal and torsional eigenfunctions. 
 However, our numerical results show that ``low" eigenvalues are simple for $\ell\ll\pi$ and $\sigma\in(0,1/2)$, furthermore ``high" modes are activated by bending energy so large that they not appear in realistic situations; it follows that eigenvalues are expected to be simple in the applications.   \par
 For future purposes it is convenient to characterize in a variational way longitudinal and torsional eigenvalues. First, for $j\in \N_+$ fixed, we introduce, respectively, the spaces $U^\mathcal{E}_j\subset H^2_\mathcal{E}$ of the first $(j-1)$ longitudinal eigenfunctions and $U^\mathcal{O}_j\subset H^2_\mathcal{O}$ of the first $(j-1)$ torsional eigenfunctions of \eqref{weight}. Then we define
 \begin{equation*}
 V^\mathcal{E}_j:=\{u\in H^2_\mathcal{E}\,:\,(u,v)_{H^2_*}=0 \quad \forall v\in U^\mathcal{E}_{j}\},\qquad V^\mathcal{O}_j:=\{u\in H^2_\mathcal{O}\,:\,(u,v)_{H^2_*}=0 \quad \forall v\in U^\mathcal{O}_{j}\}\,,
 \end{equation*}
 where if $j=1$ we mean $V^\mathcal{E}_1=H^2_\mathcal{E}$ and $V^\mathcal{O}_1=H^2_\mathcal{O}$.
Finally, using \eqref{caract2}, we set
\begin{equation}\label{mu-caract}
\mu_j(p)=\inf_{\substack{u\in V^\mathcal{E}_j\setminus\{0\} }}
\frac{\|u\|_{H^2_*}^2}{\|\sqrt{p}\,u\|_{2}^2}
\quad \text{and} \quad
\nu_j(p)=\inf_{\substack{u\in V_j^\mathcal{O} \setminus\{0\} }}
\frac{\|u\|_{H^2_*}^2}{\|\sqrt{p}\,u\|_{2}^2}\,.
\end{equation}

\section{Main results}\label{main}
As in Section \ref{problem} we always assume 
$$
0<\sigma<\frac{1}{2}\quad\text{and}\quad \alpha<1<\beta \quad (\alpha,\beta\in (0,+\infty)).
$$ 
The final goal of our analysis is to maximize the ratio \eqref{opt_intro} with the family $P_{\alpha,\beta}$ defined in \eqref{eq:famiglia}. To this aim we need first to clarify which eigenvalues we shall consider in the ratio; the model situation we have in mind is a motion concentrated on a longitudinal mode, with corresponding eigenvalue $\mu_j$ and we want to prevent the transfer of energy from this mode to the nearest torsional one $\nu_i$, for suitable $i,j\in \N_+$. Rocard \cite[p.169]{rocard} claims that, for the usual design of bridges, the eigenvalues of the observed longitudinal oscillating modes are larger than those of torsional modes, i.e. $\mu_j< \nu_i$. For the homogeneous plate this inequality readily follows from \eq{primi} if $j=i=1$. More in general, we set
\begin{equation}\label{j0}
j_0:=\max\{j\in \N_+\\: \\\nu_1(1)-\mu_{j}(1)>0 \}\,.
\end{equation}
Clearly, $j_0\geq 1$ and $j_0=j_0(\ell, \sigma)$; in our numerical experiments, for several values of $\ell$ and $\sigma$ chosen, taking into account real bridges, we always obtain $j_0=10$. As explained in \cite[Section 1]{bfg} this number is in accordance with what reported in the Federal Report \cite{ammann}, since a moment before the collapse of the Tacoma Narrows Bridge the motion was involving nine or ten longitudinal waves. Coming back to the choice of the eigenvalues in the ratio \eqref{opt_intro},  for what observed, we finally focus on the problem
\begin{equation}\label{opt1}
\mathcal{R}=\sup_{p\in P_{\alpha,\beta}}\dfrac{\nu_1(p)}{\mu_{j_0}(p)}\,.
\end{equation}
Note that if $j_0>1$, then $\nu_1(p)/\mu_{j_0}(p)\leq \nu_1(p)/\mu_{j}(p)$ for all $1\leq j< j_0$; therefore weights $p$ increasing the value of $\nu_1(p)/\mu_{j_0}(p)$ also increase the value of $\nu_1(p)/\mu_{j}(p)$ for all $1\leq j< j_0$.

First we prove
\begin{theorem}\label{existence}
Let $j_0 \in \N_+$ be as defined in \eqref{j0}. Then, problem \eqref{opt1} admits a solution. 
\end{theorem}
As already explained in the introduction, a precise theoretical characterization of maximizers to problem \eq{opt1} seems hard to reach at the current state of studies. For this reason, we concentrate our efforts in looking for weights increasing $\nu_1(p)$ or reducing $\mu_{j_0}(p)$, separately. We start by facing the problem
\begin{equation}\label{CP2}
\nu^{\alpha,\beta}_1:=\sup_{p \in P_{\alpha,\beta}} \, \nu_1(p)\,,
\end{equation}
where $\nu_1(p)$ is defined in \eqref{mu-caract} taking $j=1$. We call \textit{optimal pair} for \eqref{CP2} a couple $(\widehat{p},\widehat u)$ such that $\widehat p$ achieves the supremum in \eqref{CP2} and $\widehat u$ is an eigenfunction of $\nu_1(\widehat p)$. In the following we will always indicate with $\chi_D$ the characteristic function of a set $D\subset \R^2$.  In Section \ref{proof} we prove
\begin{theorem}\label{thm-nu}
Problem \eqref{CP2} admits an optimal pair $(\widehat{p},\widehat u) \in  P_{\alpha,\beta}\times H^2_\mathcal{O}$.
	Furthermore, $\widehat u$ and $\widehat{p}$ are related as follows
	\begin{equation*}\label{pnu}
	\widehat{p}(x,y) = \beta \chi_{ \widehat S} (x,y)+ \alpha \chi_{\Omega \setminus \widehat S}(x,y)\,\quad \text{for a.e. } (x,y)\in \Omega\,,
	\end{equation*}
	where $\widehat S = \{ (x,y)\in \Omega\,:\,\widehat u^2(x,y) \leq \widehat t \}$ for some $\widehat t> 0$ such that $|\widehat S|=\frac{1-\alpha}{\beta-\alpha}\,|\Omega|$.
\end{theorem}
Next we focus on longitudinal eigenvalues. For $j\in \N_+$, we set the minimum problem
\begin{equation}\label{CP}
\mu^{\alpha,\beta}_j :=\inf_{p \in P_{\alpha,\beta}} \, \mu_j(p)\,,
\end{equation}
where $\mu_j(p)$ is as defined in \eqref{mu-caract}.
We call \textit{optimal pair} for \eq{CP} a couple $(\overline{p}_j,{\overline{u}_j})$ such that $\overline{p}_j$ achieves the infimum in \eqref{CP} and $\overline{u}_j$ is an eigenfunction of $\mu_j(\overline{p}_j)$. When $j=1$ the counterpart of Theorem \ref{thm-nu} for problem \eqref{CP} is basically known from \cite{befafega} where the minimization issue for $\lambda_1(p)$, as defined in \eqref{first}, was dealt with. More precisely, the same proof of \cite[Theorem 3.2]{befafega} with minor changes yields the following statement:

\begin{proposition}\label{thm-mu1} \cite{befafega}
	Set $j=1$, then problem \eqref{CP} admits an optimal pair $(\overline{p}_1,{\overline{u}_1}) \in P_{\alpha,\beta}\times H^2_\mathcal{E}$.
	Furthermore, $\overline{u}_1 $ and $\overline p_1$ are related as follows
	\begin{equation*}\label{p1}
	\overline{p}_1(x,y) = \alpha \chi_{ S_1} (x,y)+ \beta \chi_{\Omega \setminus S_1}(x,y)\,\quad \text{for a.e. } (x,y)\in \Omega\,,
	\end{equation*}
	where $ S_1 = \{ (x,y)\in \Omega\,:\,\overline u_1^2(x,y) \leq t_1 \}$ for some $ t_1> 0$ such that $| S_1|=\frac{\beta-1}{\beta-\alpha}\,|\Omega|$.
\end{proposition}

Things become more involved for higher longitudinal eigenvalues. Indeed, the proofs of Theorem  \ref{thm-nu} and Proposition \ref{thm-mu1} are based on suitable rearrangement inequalities, see Lemma \ref{rearrangement} below, involving $\widehat p$ and $\overline{p}_1 $, respectively; this approach does not carry over to the case $j\geq 2$ since, in general, the orthogonality condition in the sets $V^\mathcal{E}_j$ of \eqref{mu-caract} is not preserved when changing weights. For this reason, we proceed differently and we lower $\mu_j(p)$ ``indirectly". More precisely, we first derive upper bounds for $\mu_j(p)$, where the eigenfunctions $\overline{u}_j$ are, in some sense, replaced by functions suitably chosen in $H^2_\mathcal{E}$; then we look for weights effective in lowering the upper bounds and, in turn, $\mu_j(p)$. We do not claim that this indirect approach will give the optimal density, however it suggests explicit weights effective in lowering higher eigenvalues and furthermore it provides a theoretical validation of the numerical results we collect in Section \ref{num1}.\par For $j\geq 2$ fixed and $m=1,\dots,j$, we introduce the following functions having disjoint supports  
\begin{equation}\label{sin2}
w_m(x,y):=
\begin{cases}
\sin^2(jx)\quad& \text{if }(x,y)\in\Omega^j_{m}\\
0\quad & \text{if }(x,y)\in\Omega \setminus\Omega^j_{m},
\end{cases}
\end{equation}
where $\Omega^j_{m}:=\bigg(\dfrac{(m-1)\pi}{j},\dfrac{m\pi}{j}\bigg)\times(-\ell, \ell)\subset \Omega$; it is readily checked that $w_m\in C^1(\overline \Omega)\cap  H^2_\mathcal{E}$ for all $m=1,...,j$. Then, we prove:
\begin{theorem}\label{thm-muj}
	Let $j\geq 2$ integer, then problem \eqref{CP} admits an optimal pair $ (\overline{p}_j,{\overline{u}_j}) \in P_{\alpha,\beta}\times  H^2_\mathcal{E}$
and there holds
	\begin{equation}\label{stimap}
	  \mu_j(p) \leq \inf_{p \in P_{\alpha,\beta}} \,  \left\{\max_{\substack{m=1,\dots,j}}\bigg\{\frac{1}{\|\sqrt{p}\, w_m\|^2_2}\bigg\}\right\}j^3|\Omega|\,.
	\end{equation}
	In particular, denoting by $P_{\alpha,\beta}^{per}:=\{p\in P_{\alpha,\beta}:p(x,y)=p\big(x+\frac{\pi}{j},y\big),\hspace{1mm}\text{for a.e. } (x,y)\in\Omega\}$, we have
	\begin{equation}\label{stimap1}
\mu_j(p)  \leq \inf_{p \in P^{per}_{\alpha,\beta}} \,  \left\{\frac{1}{\|\sqrt{p} \, \sin^2(jx) \|^2_2}\right\}j^4|\Omega|
	\end{equation}
	and the latter infimum is achieved by the functions
		\begin{equation*}\label{pm2}
	p_j(x,y) = \alpha \chi_{ S_j} (x,y)+ \beta \chi_{\Omega \setminus S_j}(x,y)\,\quad \text{for a.e. }(x,y)\in \Omega\,,
	\end{equation*}
	where $ S_j = \{ (x,y)\in \Omega\,:\, \sin^4(jx)  \leq t_j \}$ for $ t_j> 0$ such that $| S_j|=\frac{\beta-1}{\beta-\alpha}\,|\Omega|$.
\end{theorem}
\begin{remark}
	A comment on the choice of the functions $w_m$ is in order. The idea of taking functions $\pi/j$-periodic in the $x$-variable comes from the explicit form of the longitudinal eigenfunctions of Proposition \ref{eigenvalue}; slightly changes in the analytic expression of functions $w_m$ will qualitatively produce the same weights $p_j$, e.g. replacing $\sin^2(jx)$ with $\sin^{2n}(jx)$ ($n\geq 2$ integer) or $\exp\big[-1/(1-|x|^2)\big]$ properly rescaled and shifted in each $\Omega^j_m$. We underline that there is accordance between the optimal weights found numerically in Section \ref{num1} and the weights $p_j$ of Theorem \ref{thm-muj}.
\end{remark}

\begin{figure}[!hbt]
	\centering
	{\includegraphics[width=15cm]{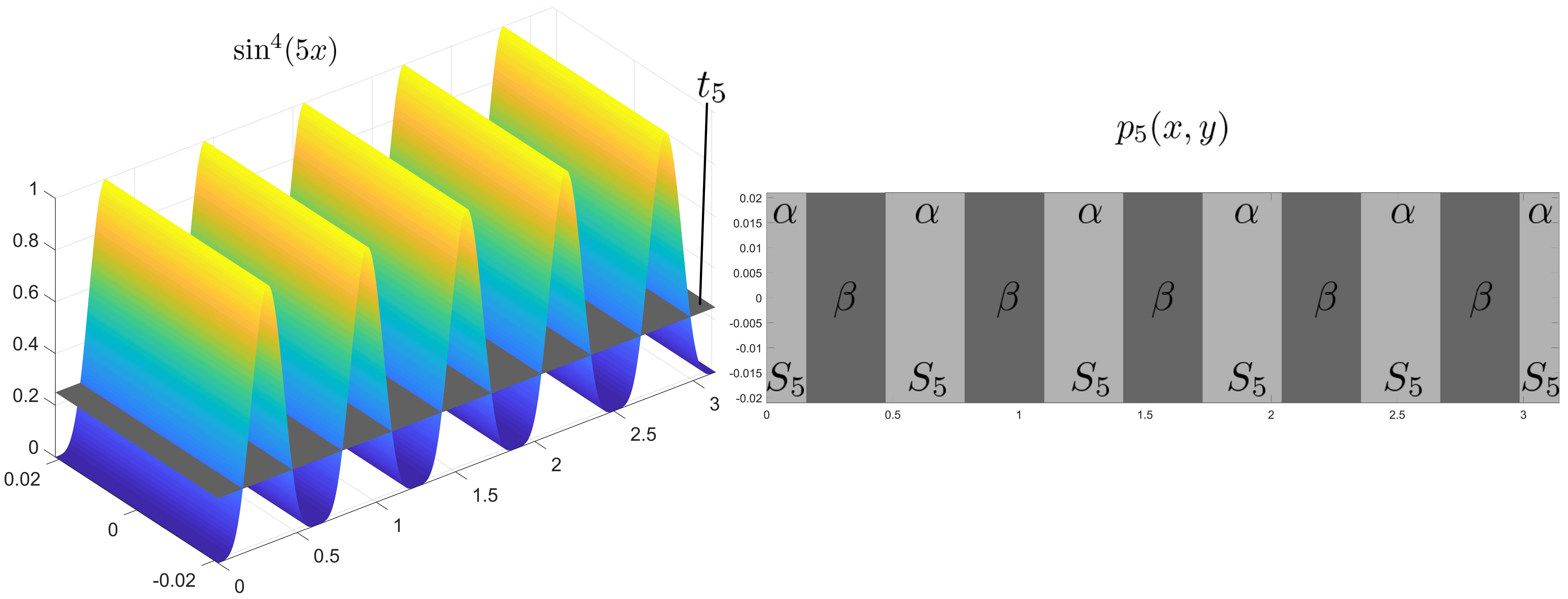}}
	\caption{Plots of $z=\sin^4(5x) $ intersected with the plane $z=t_5$ and  the correspondent set $S_5$, related to the weight $p_5(x,y)$, for a plate with $\ell=\pi/150$ ($\alpha=0.5$, $\beta=1.5$).}
	\label{fig0}
\end{figure}
We observe that, while the sets $\widehat S$ and $S_1$ of Theorem \ref{thm-nu} and Proposition \ref{thm-mu1} depend on the unknown functions $\widehat u$ and $\overline u_1$, the set $S_j$ of Theorem \ref{thm-muj} is explicitly given once determined $t_j>0$. As a matter of example, in Figure \ref{fig0} we plot the function $z=\sin^4(5x)$, the corresponding set $S_5$ and the related weight $p_5(x,y)$. It is worth noting that the statement of Theorem \ref{thm-nu} combines nicely with those of  Proposition \ref{thm-mu1} and Theorem \ref{thm-muj} in increasing the ratio in \eqref{opt1}. This is highlighted by the numerical experiments we collect in Section \ref{num2}.

\section{Numerical results}\label{num2}
In the previous section we proved that an optimal weight maximizing the ratio $\nu_1(p)/\mu_{j_0}(p)$, with $j_0$ defined in \eqref{j0}, exists. Then, in order to find information on its analytic expression, we decided to minimize $\mu_{j_0}(p)$ or maximize $\nu_1(p)$, separately. All the theoretical results obtained tell that the optimal weights $p\in P_{\alpha,\beta}$, either for problem \eq{CP2} and \eq{CP}, must be of \textit{bang-bang} type, i.e.
$$
p(x,y) = \alpha \chi_{S} (x,y)+ \beta \chi_{\Omega \setminus S}(x,y)\,\quad \text{for a.e. } (x,y)\in \Omega\,,
$$
for a suitable set $S\subset \Omega$. In other words, the plate must be composed by two different materials properly located in $\Omega$; this is useful in engineering terms, since the assemblage of two materials with constant density is simpler than the manufacturing of a material having variable density. Unfortunately, Theorem \ref{thm-nu} and Proposition \ref{thm-mu1} give no precise information on the location of the set $S$; nevertheless, through suitable numerical experiments, we are able to suggest what could be the optimal design of the set $S$, in problems \eq{CP2} and \eq{CP}, and to guess a possible maximizer to problem \eq{opt1}.
For the applicative purpose we may strengthen the plate with steel and  we may consider the other material composed by a mixture of steel and concrete; following this approach, the denser material has approximately triple density with respect to the weaken, i.e. $\beta=3\alpha$.
\subsection{Eigenvalues computation}
We propose a numerical method to find approximate solutions of \eqref{weight} which relies on the explicit information we have from Proposition \ref{eigenvalue} ($p\equiv1$).
%Let $m$ be a positive integer, we introduce the scalar product in $H^2(-\ell, \ell)$:
%$$
%\langle\varphi,\eta\rangle_m:=\int_{-\ell}^{\ell}\left( \varphi'' \eta''+2m^2(1-\sigma) \varphi' \eta'-\sigma m^2 (\varphi''\eta+\varphi\eta'')+m^4 \varphi \eta\right) \,dy \, ,
%$$
%and we denote by $(\cdot,\cdot)_2$ the usual scalar product in $L^2(-\ell,\ell)$. 
%Expanding the solution of \eqref{weight} with $p\equiv 1$ in Fourier series $
%u(x,y)=\sum_{m=1}^{+\infty} \varphi_m(y)\sin(mx)
%$, we get that the Fourier coefficients $\varphi_m(y)$ solve, for every $m\geq 1$ fixed, the following differential problem in weak form
%\begin{equation}\label{weight1dweak0}
%\langle\varphi_m,\eta\rangle_m=\lambda(\varphi_m, \eta)_2 \quad \forall\eta \in H^2(-\ell, \ell),
%\end{equation}
%which corresponds to 
%\begin{equation}\label{weight1d0}
%\begin{cases}
%\varphi''''(y)-2m^2\varphi''(y)+m^4\varphi(y)=\lambda \varphi(y) & \qquad \text{in } (-\ell,\ell) \\
%\varphi''(\pm\ell)-\sigma m^2\varphi(\pm\ell)=0 & \qquad  \\
%\varphi'''(\pm\ell)-(2-\sigma)m^2\varphi'(\pm\ell)=0\,,& \qquad \,
%\end{cases}
%\end{equation}
%see also \cite{befafega}.\par 
We expand the solutions $u$ of \eqref{weight} in Fourier series, adopting the orthonormal basis of $L^2$ given by eigenfunctions of the homogeneous plate. More precisely, denoting by $z_m(x,y)\in H^2_\mathcal{E}$ and $\theta_m(x,y)\in H^2_\mathcal{O}$, respectively, the (ordered) longitudinal and torsional eigenfunctions of problem \eqref{weight} with $p\equiv1$, $u$ writes
\begin{equation}\label{fourier}
u(x,y)=\sum_{m=1}^{\infty}\left[a_{m}z_{m}(x,y)+b_{m}\theta_{m}(x,y)\right]\,,
\end{equation}
for suitable $a_m,b_m \in \R$.
%We assume for a while that $p\in P_{\alpha,\beta}$, but with $p(x,y)$ not necessarily $y$-even.
In order to get a numerical approximation, we trunk the series in \eqref{fourier} at $N \in \N_{+}$ and we
plug the Fourier sum into \eqref{eigenweak}. We recall that, for all $m\in \mathbb{N}_+$, $z_{m}$ and $\theta_{m}$ solve:
\begin{equation*}\label{base1d}
\begin{split}
(z_{m},v)_{H^2_*}=\mu_{m}(1)\, (z_{m}, v)_{L^2} \quad &\forall v \in H^2_*\,\\
(\theta_{m},v)_{H^2_*}=\nu_{m}(1)\,(\theta_{m}, v)_{L^2} \quad &\forall v \in H^2_*\,,
\end{split}
\end{equation*}
where $\mu_{m}(1)$ and $\nu_{m}(1)$ are defined in \eq{mu-caract} with $p\equiv 1$.
Therefore, we obtain the following finite dimensional linear system in the unknowns $a_{n}$ and $b_{n}$:
\begin{equation}\label{weightnum3}
\begin{cases}
a_{n}\mu_{n}(1)=\mu(p) \sum\limits_{m=1}^{N}\, a_{m}\,C^p_{n,m}\\
b_{n}\nu_{n}(1)=\nu(p) \sum\limits_{m=1}^{N}\, b_{m}\,\overline{C}^p_{n,m},
\end{cases}
\end{equation}
\begin{flushright}
	for $n=1,\dots,N$
\end{flushright}
where
\begin{equation*}
\begin{split}
&C^p_{n,m}:=\int_{\Omega}p(x,y)\, z_{n}(x,y)\,z_m(x,y) \,dxdy\\&
\overline C^p_{n,m}:=\int_{\Omega}p(x,y)\, \theta_{n}(x,y)\,\theta_m(x,y) \,dxdy .
\end{split}
\end{equation*}
In particular, by solving \eq{weightnum3}, it is possible to determine $N$ approximated longitudinal eigenvalues $\mu_n(p)$ and $N$ torsional eigenvalues $\nu_n(p)$. We observe that the decoupling between the unknowns $a_{n}$ and $b_{n}$, which produces eigenfunctions even or odd in $y$, is due to the assumption on $p\in P_{\alpha,\beta}$,  being $y$-even.\par 
In order to compute numerically the eigenvalues $\mu_n(p)$ and $\nu_n(p)$ for suitable choices of the weight $p$, we fix from now onward:
\begin{equation}\label{parameter}
\sigma=0.2\qquad \text{and}\qquad\ell=\dfrac{\pi}{150},
\end{equation}
which is a choice consistent with common bridge design. The explicit values of $\mu_n(1)$ and $\nu_n(1)$ are computed by exploiting Proposition  \ref{eigenvalue}, see Table \ref{eig num}. When condition \eq{parameter} holds, we numerically find that the eigenvalues $\Lambda^{m,1}$ do not exist for $1\leq m\leq 2734$ and that
\begin{equation*}\label{modi}
\begin{split}
\mu_{m}(1)=\Lambda_{m,1} \quad &{\rm for}\quad m=1,\dots,113\\
\nu_{m}(1)=\Lambda^{m,2} \quad &{\rm for}\quad m=1,\dots,174\,.
\end{split}
\end{equation*}
Hence, by Proposition \ref{eigenvalue} we know that the basis of eigenfunctions exploited in \eqref{fourier} writes $z_m(x,y)=\phi_{m,1}(y)\sin(mx)$ and $\theta_m(x,y)=\psi_{m,2}(y)\sin(mx)$ for $1\leq m\leq 113$.

\begin{table}[h]\centering
	\scalebox{0.9}{	\begin{tabular}{|c||c|}
			\hline
			&$\mu_{m}(1)=\Lambda_{m,1}$\\%&$k=2$\\
			\hline
			\hline
			$m=1$&9.60$\cdot 10^{-1}$\\%&$\mu_{114}(1)$=1.6264$\cdot 10^8$\\
			$m=2$&1.54$\cdot 10^1$\\%&$\mu_{115}(1)$=1.6279$\cdot 10^8$\\
			$m=3$&7.78$\cdot 10^1$\\%&$\mu_{116}(1)$=1.6303$\cdot 10^8$\\
			$m=4$&2.46$\cdot 10^2$\\%&$\mu_{117}(1)$=1.6336$\cdot 10^8$\\
			$m=5$&6.00$\cdot 10^2$\\%&$\mu_{118}(1)$=1.6379$\cdot 10^8$\\
			$m=6$&1.24$\cdot 10^3$\\%&$\mu_{119}(1)$=1.6432$\cdot 10^8$\\
			$m=7$&2.31$\cdot 10^3$\\%&$\mu_{120}(1)$=1.6495$\cdot 10^8$\\
			$m=8$&3.93$\cdot 10^3$\\%&$\mu_{121}(1)$=1.6567$\cdot 10^8$\\
			$m=9$&6.30$\cdot 10^3$\\%&$\mu_{122}(1)$=1.6648$\cdot 10^8$\\
			$m=10$&9.61$\cdot 10^3$\\%&$\mu_{124}(1)$=1.6740$\cdot 10^8$\\
			$m=11$&1.41$\cdot 10^4$\\%&$\mu_{125}(1)$=1.6841$\cdot 10^8$\\
			$m=12$&1.99$\cdot 10^4$\\%&$\mu_{126}(1)$=1.6951$\cdot 10^8$\\
			%	$m=13$&27461.60&1.7072$\cdot 10^8$&4.7808$\cdot 10^9$\\
			%	$m=14$&36946.00&1.7202$\cdot 10^8$&4.7860$\cdot 10^9$\\
			%	$m=15$&48700.00&1.7342$\cdot 10^8$&4.7916$\cdot 10^9$\\
			\hline
	\end{tabular}}\qquad\quad
	\scalebox{0.9}{	\begin{tabular}{|c||c|}
			\hline
			&$\nu_m(1)=\Lambda^{m,2}$\\%&$k=3$\\
			\hline
			\hline
			$m=1$&1.09$\cdot 10^4$\\%&$\nu_{175}(1)$=1.2356$\cdot 10^9$\\
			$m=2$&4.38$\cdot 10^4$\\%&$\nu_{176}(1)$=1.2359$\cdot 10^9$\\
			$m=3$&9.86$\cdot 10^4$\\%&$\nu_{177}(1)$=1.2359$\cdot 10^9$\\
			$m=4$&1.75$\cdot 10^5$\\%&$\nu_{178}(1)$=1.2372$\cdot 10^9$\\
			$m=5$&2.74$\cdot 10^5$\\%&$\nu_{179}(1)$=1.2382$\cdot 10^9$\\
			$m=6$&3.95$\cdot 10^5$\\%&$\nu_{180}(1)$=1.2394$\cdot 10^9$\\
			$m=7$&5.38$\cdot 10^5$\\%&$\nu_{181}(1)$=1.2409$\cdot 10^9$\\
			$m=8$&7.04$\cdot 10^5$\\%&$\nu_{182}(1)$=1.2425$\cdot 10^9$\\
			$m=9$&8.93$\cdot 10^5$\\%&$\nu_{183}(1)$=1.2444$\cdot 10^9$\\
			$m=10$&1.10$\cdot 10^6$\\%&$\nu_{184}(1)$=1.2465$\cdot 10^9$\\
			$m=11$&1.34$\cdot 10^6$\\%&$\nu_{185}(1)$=1.2488$\cdot 10^9$\\
			$m=12$&1.60$\cdot 10^6$\\%&$\nu_{186}(1)$=1.2513$\cdot 10^9$\\
			%	$m=13$&1.8762$\cdot 10^6$&1.2541$\cdot 10^9$&1.3025$\cdot 10^{10}$\\
			%	$m=14$&2.1810$\cdot 10^6$&1.2570$\cdot 10^9$&1.3033$\cdot 10^{10}$\\
			%	$m=15$&2.5098$\cdot 10^6$&1.2603$\cdot 10^9$&1.3042$\cdot 10^{10}$\\
			\hline
	\end{tabular}}
	\vspace{3mm}
	\caption{On the left the lowest longitudinal eigenvalues $\mu_m(1)$ and on the right the lowest torsional eigenvalues $\nu_m(1)$ of \eqref{weight} with $p\equiv1$.}
	\label{eig num}
\end{table}

\subsection{Numerical solution of \eq{CP}}\label{num1} Let us begin by minimizing $\mu_1(p)$, i.e. the first longitudinal eigenvalue of problem \eqref{weight}, as characterized in \eqref{mu-caract} with $j=1$. In order to find the optimal weight given by Proposition \ref{thm-mu1}, we adopt a numerical algorithm proposed in \cite{chen} that we shortly illustrate in the following.  First we solve numerically \eqref{weightnum3} with a given weight $p^{(i)}$ and we determine the corresponding eigenfunction $u_1^{(i)}$. Then, we choose a weight at the next iteration $p^{(i+1)}$ such that
\begin{equation*}\label{algoritmo}
\|\sqrt{p^{(i+1)}}u_1^{(i)}\|_2^2\geq \|\sqrt{p^{(i)}}u_1^{(i)}\|_2^2,
\end{equation*}
in order to have
\begin{equation*}
\mu^{(i+1)}_1= \min_{\substack{u \in H^2_\mathcal{E}\setminus\{0\}}}
\frac{\|u\|_{H^2_*}^2}{\|\sqrt{p^{(i+1)}}\,u\|_{2}^2}=\frac{\|u_1^{(i+1)}\|_{H^2_*}^2}{\|\sqrt{p^{(i+1)}}\,u_1^{(i+1)}\|_{2}^2}\leq \frac{\|u_1^{(i)}\|_{H^2_*}^2}{\|\sqrt{p^{(i+1)}}\,u_1^{(i)}\|_{2}^2}\leq \frac{\|u_1^{(i)}\|_{H^2_*}^2}{\|\sqrt{p^{(i)}}\,u_1^{(i)}\|_{2}^2}=\mu^{(i)}_1.
\end{equation*}
Note that to select $p^{(i+1)}$ we exploited the rearrangement Lemma \ref{rearrangement} below.
Iterating, we obtain a decreasing sequence of eigenvalues; since the infimum in \eqref{CP} is achieved, the sequence is bounded from below by  $\mu_1^{\alpha, \beta}$ so that it is convergent. We stop the algorithm when $|\mu_1^{(i+1)}-\mu_1^{(i)}|<\epsilon$, with $\epsilon=10^{-4}\div10^{-3}$. As pointed out in \cite{chanillo} it is not clear a priori if the sequence converges to $\mu_1^{\alpha,\beta}$ or not; to avoid the latter case we repeated the procedure considering different weights at the first iteration and we always obtain the convergence to the same values.\par In Figure \ref{fig1} we plot the set $S_1$ defined in Proposition \ref{thm-mu1} for the eigenfunction $\overline u_1$ of the obtained numerical optimal pair; clearly, the direction is to concentrate the denser material near the maximum of $\overline u_1^2$. Since the set $\Omega\setminus S_1$ is similar to a rectangle, we propose the following analytic expression of the \textbf{approximated optimal weight }for $\mu_1^{\alpha, \beta}$:  

\begin{equation*}\label{p11}
\overline p_1(x,y)=\overline p_1(x):=\beta \chi_{I_1}(x)+\alpha \chi_{(0,\pi)\setminus I_1}(x)\,\quad \text{for a.e. } (x,y)\in \Omega\,,
\end{equation*}
where $I_1:=\left(\frac{\pi}{2}-\frac{\pi}{2}\frac{(1-\alpha)}{(\beta-\alpha)},\,\frac{\pi}{2}+\frac{\pi}{2}\frac{(1-\alpha)}{(\beta-\alpha)} \right)$.

\begin{figure}[!hbt]
	\centering
	{\includegraphics[width=15cm]{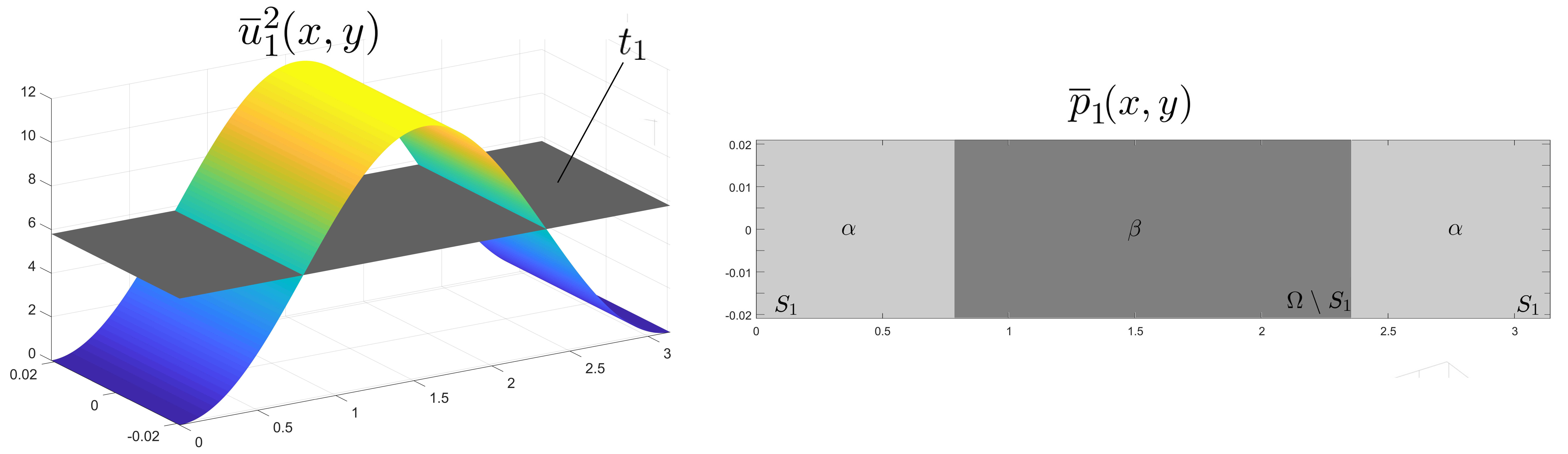}}
	\caption{Plot of $z=\overline u_1^2(x,y)$ intersected with the plane $z=t_1$ and plot of the related set $S_1$ ($\alpha=0.5$, $\beta=1.5$, $N=20$).}
	\label{fig1}
\end{figure}

The previous algorithm can be adapted to determine $\mu_j(\overline p_j)$ for generic $j\in\mathbb{N}_+$; in particular, if $j>1$ we apply the characterization \eqref{caract2} of eigenvalues, i.e. we consider the minimum onto the space $V^\mathcal{E}_j$ instead of $H^2_\mathcal{E}$. The further difficulty is that, now, at the end of every iteration, we have to check that $u_j^{(i)}(x,y)$ does not belong to the subspace spanned by $\{u_1^{(i+1)},\dots, u_{j-1}^{(i+1)}\}$, for more details see \cite{chen}.  For each $j\in\mathbb{N}_+$, the obtained optimal weight has the denser material concentrated near to the peaks of the associated eigenfunction which are, approximatively, located at $\frac{\pi}{2j}(2h-1)$ with $h=1,...,j$; this is aligned with the statement of Theorem \ref{thm-muj}. Therefore, we propose the following \textbf{approximated optimal weight }for $\mu_j^{\alpha, \beta}$:
\begin{equation}\label{pj}
\overline p_j(x,y)=\overline p_j(x):=\beta \chi_{I_j}(x)+\alpha \chi_{(0,\pi)\setminus I_j}(x)\\,\quad \text{for a.e. } (x,y)\in \Omega\,,
\end{equation}
where $I_j:=\displaystyle{\bigcup_{h=1}^j}\bigg(\frac{\pi}{2j}(2h-1)-\frac{\pi}{j} \frac{(1-\alpha)}{2(\beta-\alpha)},\,\frac{\pi}{2j}(2h-1)+\frac{\pi}{j} \frac{(1-\alpha)}{2(\beta-\alpha)}\bigg)$.\par
The above results show that the optimal weight changes if we change $j$. Nevertheless, numerically, we observe that the weight $\overline p_{j}(x)$ in \eqref{pj} reduces not only $\mu_{j}(\overline p_{j})$, but also all the previous longitudinal eigenvalues $\mu_i(\overline p_{j})$ with $1\leq i<j$; while it increases $\mu_i(\overline p_{j})$ with $i>j$. This means that, if it were possible to predict the highest mode of vibration for a plate during its design, then there would be an optimal reinforce for it, reducing at the same time all the previous ones.

\subsection{Numerical solution to \eq{CP2}} \label{sub2}
About the maximization of the first torsional eigenvalue, we cannot adapt the algorithm in \cite{chen} that only works for infimum problems.
Nevertheless, Theorem \ref{thm-nu} suggests to put the denser material in the region $\widehat S = \{ (x,y)\in \Omega\,:\,\widehat u^2(x,y) \leq \widehat t \}$ for some $\widehat t> 0$, where $\widehat{u}$ is the eigenfunction corresponding to $\nu_1(\widehat{p})$. Since we do not know explicitly $\widehat u$, we proceed by trial and error; we start by replacing $\widehat u$ with the first torsional eigenfunction $\theta_1(x,y)=\psi_{1,2}(y)\sin(x)$ of problem \eqref{weight} with $p\equiv1$ and we define the weight $p^*(x,y):= \beta \chi_{S^*}(x,y)+ \alpha \chi_{\Omega \setminus S^*}(x,y)$ where $S^* := \{ (x,y)\in \Omega\,:\, \theta_1^2(x,y) \leq t^* \}$, for $t^*> 0$ such that $| S^*|=\frac{1-\alpha}{\beta-\alpha}\,|\Omega|$.
Then, we proceed by solving \eqref{weight} with $p^*$, obtaining a new first torsional eigenfunction $u^*$ to which we associate, as done for $\theta_1$, a new weight $p^{**}$ of bang-bang type; iterating the procedure, we observe that the obtained weights are always very close to $p^{*}$, so that we conjecture that the theoretical optimal weight $\widehat{p}$ of Theorem \ref{thm-nu} is qualitatively very similar to $p^*$. In Figure \ref{fig2} we plot $S^*$ and in Table \ref{tab2} we give the corresponding eigenvalues. 
\begin{figure}[!hbt]
	\centering
	{\includegraphics[width=15cm]{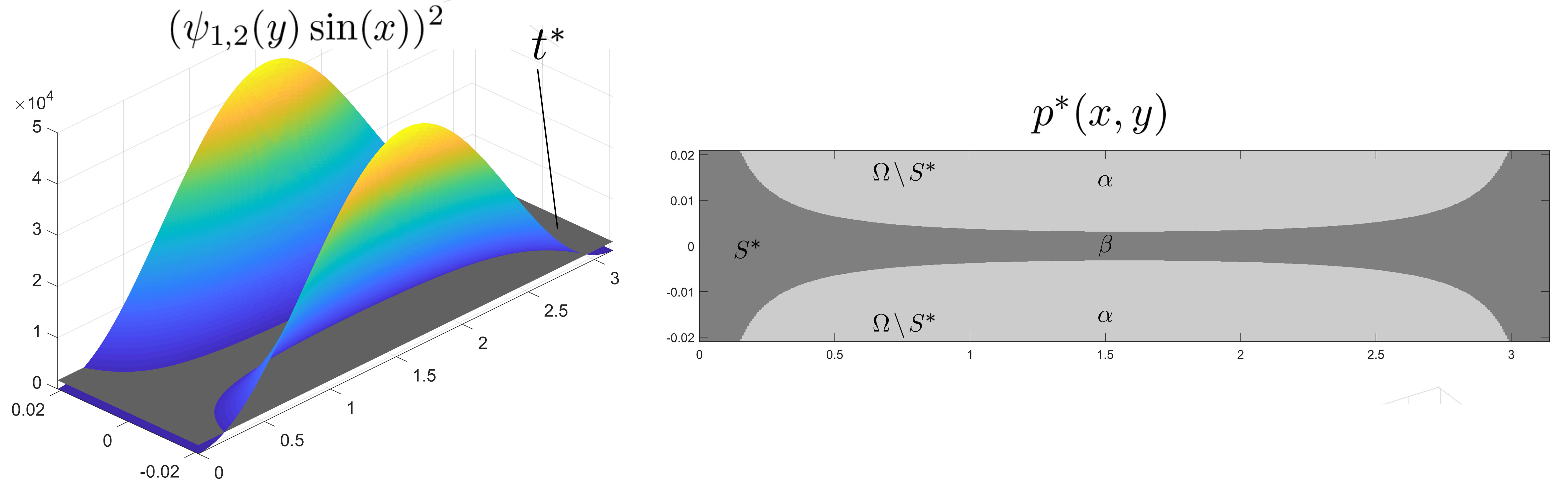}}
	\caption{Plot of $z=\theta_1^2(x,y)= (\psi_{1,2}(y)\sin(x))^2$ intersected with the plane $z=t^*$ and plot of the related set $S^*$ ($\alpha=0.5$, $\beta=1.5$).}
	\label{fig2}
\end{figure}

In order to find a reinforce more suitable for practical reproduction, inspired by Figure \ref{fig2}, we consider in our experiments a second weight depending only on $y$ and concentrated around the mid-line $y=0$, i.e. 
\begin{equation*}
\breve p(x,y)=\breve p(y):=\beta \chi_{\breve I}(y)+\alpha \chi_{(-\ell,\ell)\setminus \breve I}(y)\,\quad \text{for a.e. } (x,y)\in \Omega\,,
\end{equation*}
where $\breve I:=\big(-\frac{\ell(\beta-1)}{\beta-\alpha},\frac{\ell(\beta-1)}{\beta-\alpha}\big)$. Clearly, this choice produces some simplifications in the problem and the coefficients in \eqref{weightnum3} become simpler, see also \cite[Section 4]{befafega}. The obtained eigenvalues are again collected in Table \ref{tab2}. \par  Since the weight $\breve p(x,y)$ increases $\nu_1(p)$ less than $p^*(x,y)$, we keep as \textbf{approximated optimal weight }for $\nu_1^{\alpha, \beta}$:
$$
p^*(x,y)= \beta \chi_{S^*}(x,y)+ \alpha \chi_{\Omega \setminus S^*}(x,y) \,\quad \text{for a.e. } (x,y)\in \Omega\,,
$$
where $S^* = \{ (x,y)\in \Omega\,:\, (\theta_{1})^2(x,y) \leq t^* \}$ for $t^*> 0$ such that $|S^*|=\frac{1-\alpha}{\beta-\alpha}|\Omega|$.

\subsection{Numerical solution of \eq{opt1}}
\begin{table}[h]\centering
	\scalebox{0.9}{	\begin{tabular}{|c||c|c|c|c|c|c|}
			\hline
			&$p\equiv1$&$\overline p_{10}(x)$&$p^*(x,y)$&$\breve p(y)$&$\overline{\overline p}(x)$&$\widetilde{p}(x,y)$\\
			&\includegraphics[scale=0.22]{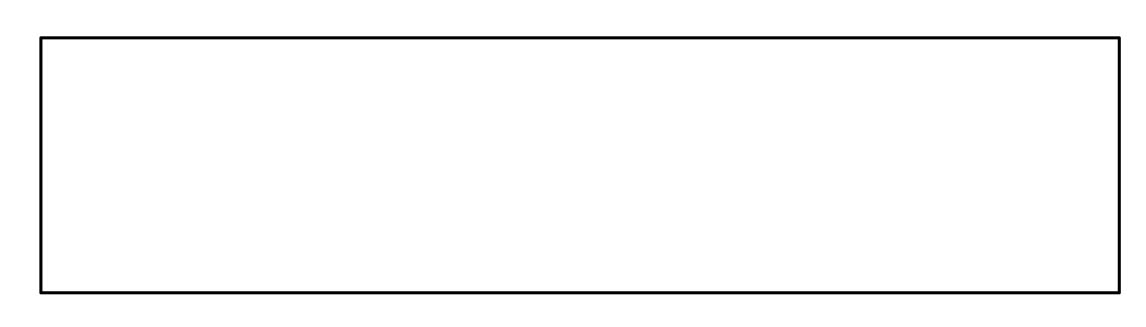}&\includegraphics[scale=0.22]{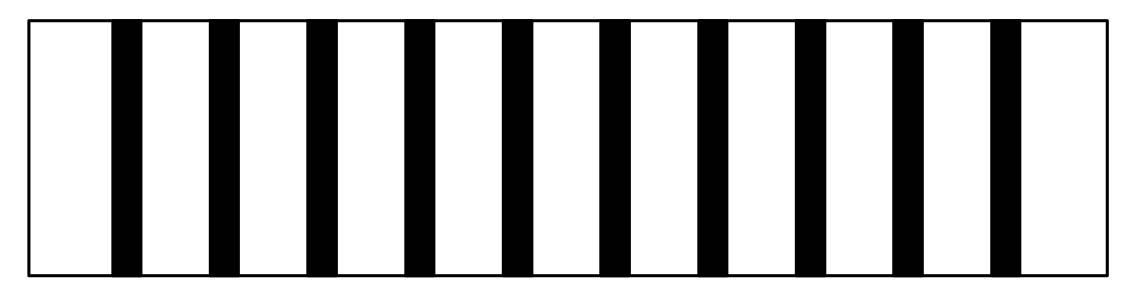}&\includegraphics[scale=0.22]{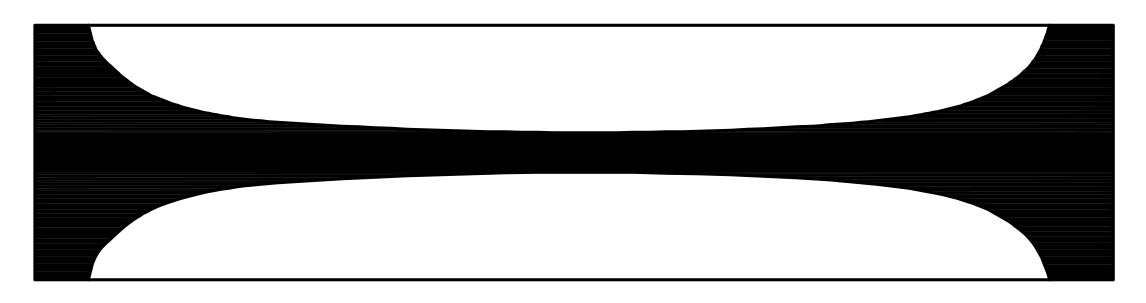}&\includegraphics[scale=0.22]{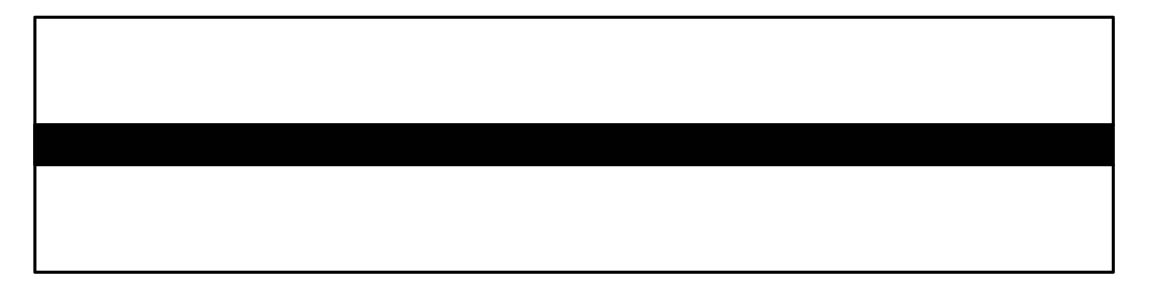}&\includegraphics[scale=0.22]{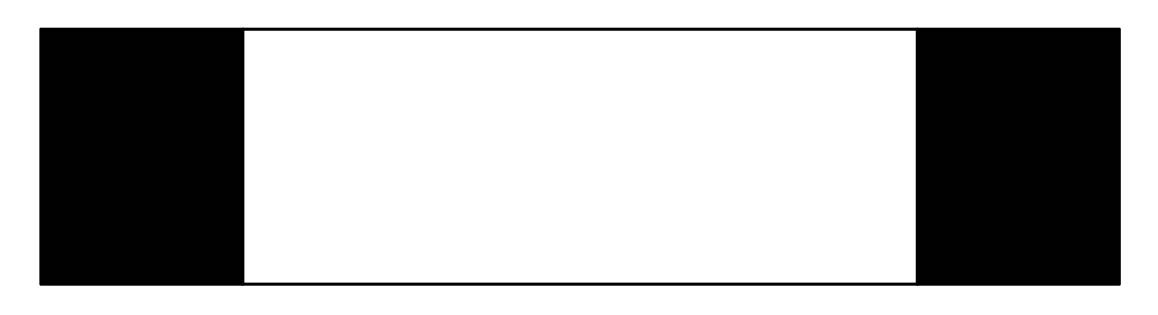}&\includegraphics[scale=0.22]{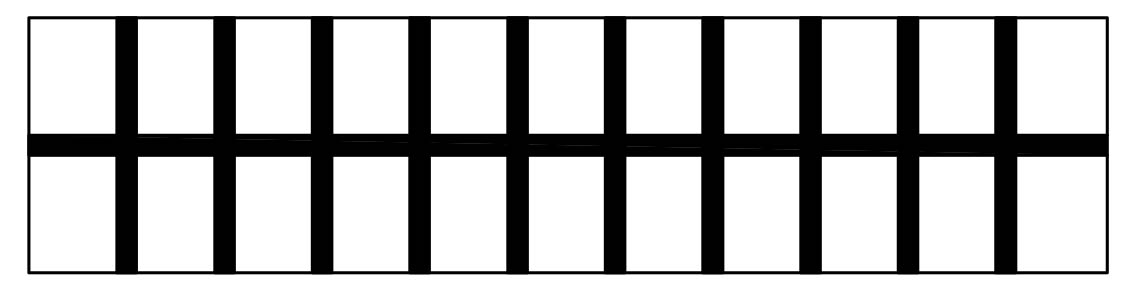}\\
			\hline
			\hline
			$\mu_1(p)$&9.60$\cdot 10^{-1}$&9.60$\cdot 10^{-1}$&1.16&9.60$\cdot 10^{-1}$&1.40&9.86$\cdot 10^{-1}$\\
			$\mu_2(p)$&1.54$\cdot 10^1$&1.54$\cdot 10^1$&1.66$\cdot 10^1$&1.54$\cdot 10^1$&1.52$\cdot 10^1$&1.58$\cdot 10^1$\\
			$\mu_3(p)$&7.78$\cdot 10^1$&7.77$\cdot 10^1$&8.06$\cdot 10^1$&7.78$\cdot 10^1$&8.05$\cdot 10^1$&7.98$\cdot 10^1$\\
			$\mu_4(p)$&2.46$\cdot 10^2$&2.46$\cdot 10^2$&2.51$\cdot 10^2$&2.46$\cdot 10^2$&2.96$\cdot 10^2$&2.52$\cdot 10^2$\\
			$\mu_5(p)$&6.00$\cdot 10^2$&5.99$\cdot 10^2$&6.10$\cdot 10^2$&6.01$\cdot 10^2$&6.78$\cdot 10^2$&6.16$\cdot 10^2$\\
			$\mu_6(p)$&1.24$\cdot 10^3$&1.24$\cdot 10^3$&1.27$\cdot 10^3$&1.25$\cdot 10^3$&1.31$\cdot 10^3$&1.28$\cdot 10^3$\\
			$\mu_7(p)$&2.31$\cdot 10^3$&2.28$\cdot 10^3$&2.36$\cdot 10^3$&2.31$\cdot 10^3$&2.60$\cdot 10^3$&2.37$\cdot 10^3$\\
			$\mu_8(p)$&3.93$\cdot 10^3$&3.84$\cdot 10^3$&4.04$\cdot 10^3$&3.94$\cdot 10^3$&4.55$\cdot 10^3$&4.04$\cdot 10^3$\\
			$\mu_9(p)$&6.30$\cdot 10^3$&5.87$\cdot 10^3$&6.48$\cdot 10^3$&6.31$\cdot 10^3$&6.85$\cdot 10^3$&6.47$\cdot 10^3$\\
			$\mu_{10}(p)$&9.61$\cdot 10^3$&7.28$\cdot 10^3$&9.90$\cdot 10^3$&9.62$\cdot 10^3$&1.04$\cdot 10^4$&9.55$\cdot 10^3$\\
			$\mu_{11}(p)$&1.41$\cdot 10^4$&1.68$\cdot 10^4$&1.45$\cdot 10^4$&1.41$\cdot 10^4$&1.61$\cdot 10^4$&1.45$\cdot 10^4$\\
			$\mu_{12}(p)$&1.99$\cdot 10^4$&2.27$\cdot 10^4$&2.05$\cdot 10^4$&2.00$\cdot 10^4$&2.24$\cdot 10^4$&2.05$\cdot 10^4$\\
			\hline
			$\nu_1(p)$&1.09$\cdot 10^4$&1.09$\cdot 10^4$&1.98$\cdot 10^4$&1.75$\cdot 10^4$&1.56$\cdot 10^4$&1.71$\cdot 10^4$\\
			$\nu_2(p)$&4.38$\cdot 10^4$&4.37$\cdot 10^4$&6.88$\cdot 10^4$&7.01$\cdot 10^4$&4.14$\cdot 10^4$&6.84$\cdot 10^4$\\
			\hline
			$\mathcal{R}$&\textbf{1.14}&\textbf{1.50}&\textbf{2.00}&\textbf{1.82}&\textbf{1.49}&\textbf{1.79}\\
			\hline
	\end{tabular}} 
	
	\vspace{3mm}
	\caption{The lowest longitudinal eigenvalues $\mu_{j}(p)$ with $j=1,\dots,12$, the first two torsional $\nu_{i}(p)$ with $i=1,2$ and the ratio $\mathcal{R}=\frac{\nu_1(p)}{\mu_{10}(p)}$ of \eqref{weightnum3} with different weights, assuming \eqref{parameter}, $\alpha=0.5$, $\beta=1.5$ and $N=30$.}
	\label{tab2}
\end{table} 
From the eigenvalues in Table \ref{eig num} we infer $$\mu_1(1)<...<\mu_{10}(1)<\nu_{1}(1)<\mu_{11}(1)<...$$
This is the reason why, we fix $j_0=10$ in \eqref{j0} and we focus on the ratio $\mathcal{R}$ between $\nu_1(p)$ and $\mu_{10}(p)$. In order to increase $\mathcal{R}$ we test weights raising $\nu_1(p)$ and lowering $\mu_{10}(p)$. \par
First we consider the optimal weight $\overline p_{10}(x)$ as defined in \eqref{pj}. As we can see from Table \ref{tab2}, it has a limited effect on the variation of $\nu_1(p)$, so that it makes sense to minimize $\mu_{10}(p)$ in order to increase the ratio \eqref{opt1}. 
\par
Next we consider weights having strong effects on $\nu_1(p)$, such as the weights $p^*(x,y)$ and $\breve p(y)$ defined in Section \ref{sub2}. Table \ref{tab2} highlights that they have a confined effect on longitudinal eigenvalues. Moreover, they increase the ratio $\mathcal{R}$ much more than the weights optimal for the longitudinal modes.\par
We complete the numerical experiments by testing other weights which seem to be reasonable in order to increase $\mathcal{R}$; more precisely, we consider a weight concentrated near the short edges of the plate:
$$\overline{\overline p}(x,y)=\overline{\overline p}(x):=\alpha \chi_{I}(x)+\beta \chi_{(0,\pi)\setminus I}(x) \,\quad \text{for a.e. } (x,y)\in \Omega\,,$$
where $I:=\big(\frac{\pi}{2}-\frac{\pi(\beta-1)}{2(\beta-\alpha)},\frac{\pi}{2}+\frac{\pi(\beta-1)}{2(\beta-\alpha)}\big)$, and the cross-type weight $\widetilde{p}(x,y)$, given in the last column of Table \ref{tab2}, which is obtained by combining $\overline p_{10}(x)$ and $\breve p(y)$. From Table \ref{tab2} we observe that these weights have effects both on torsional and on longitudinal eigenvalues, so that they do not seem optimal for the ratio $\mathcal{R}$.\par
Summing up, since $p^*(x,y)$ increases the ratio $\mathcal{R}$ more than all the other considered weights, we propose as \textbf{approximated optimal weight} for $\mathcal{R}$:
$$
p^*(x,y)= \beta \chi_{S^*}(x,y)+ \alpha \chi_{\Omega \setminus S^*}(x,y) \,\quad \text{for a.e. } (x,y)\in \Omega\,,
$$
where $S^* = \{ (x,y)\in \Omega\,:\, \theta_1^2(x,y) \leq t^* \}$ for $t^*> 0$ such that $|S^*|=\frac{1-\alpha}{\beta-\alpha}|\Omega|$, cfr. Figure \ref{fig2}.\par
%A possible drawback in increasing too much $\nu_1(p)$ is that the ratio to consider with respect to \eqref{j0}, involves higher longitudinal modes different from that of the homogeneous case, e.g. $\mu_{11}(p)$. Anyway, this fact is not so negative, since in real situation vibrations on higher modes are less frequent, and possibly improbable, with respect to vibrations on lower modes.\par
Although the present work is focused on the first torsional eigenvalue, in Table \ref{tab2} we also collect the results obtained for the second torsional eigenvalue $\nu_2(p)$. We observe that the variation of $\nu_2(p)$ follows the same trend of that of $\nu_1(p)$ with respect to the considered weights, 
%there are no significative differences between the two cases,
 hence we may conjecture that the same reinforcement could be adopted to optimize ratios involving subsequent (low) torsional eigenvalues.\par

\section{Proofs}\label{proof}
In what follows we will always assume that
$$
0<\sigma<\frac{1}{2}\quad\text{and}\quad \alpha<1<\beta \quad (\alpha,\beta\in (0,+\infty))
$$ 
and the family $P_{\alpha,\beta}$ is as defined in \eqref{eq:famiglia}.

%\subsection{Proof of Proposition \ref{eigenvalue_compare}} If the first eigenfunction $u_1(x,y)$ of \eqref{weight} is positive in $\Omega$, then $u_1\notin H^2_\mathcal{O}$, therefore the first eigenvalue corresponding to $u_1$ is longitudinal, i.e. $\mu_1(p)$.\par 
%By \eqref{mu-caract} with $j=1$ we have
%$$\mu_1(p)=\min_{\substack{u\in H^2_\mathcal{E}, u\neq 0}}\frac{\|u\|^2_{H^2_*}}{\|\sqrt{p}u\|^2_2}\leq \frac{\|u_{1,1}\|^2_{H^2_*}}{\alpha\|u_{1,1}\|^2_2}=\frac{\mu_{1}(1)}{\alpha}.$$
%Since $p\in P_{\alpha,\beta}$, we obtain
%\begin{equation*}\label{stima11}
%\|\sqrt{p}u\|^2_{2}=\int_{\Omega}pu^2\,dxdy\leq\beta\|u\|^2_{2},
%\end{equation*}
%implying by \eqref{mu-caract} with $j=1$
%%\begin{equation*}\label{stima21}
%%\frac{\|u\|^2_{H^2_*}}{\|\sqrt{p}u\|^2_2}\geq\frac{\|u\|^2_{H^2_*}}{\|p\|_{\infty}\|u\|^2_{2}}.
%%\end{equation*}
%$$\nu_1(p)=\min_{\substack{u\in H^2_\mathcal{O}}}\frac{\|u\|^2_{H^2_*}}{\|\sqrt{p}u\|^2_{2}}\geq \frac{1}{\beta}\min_{\substack{u\in H^2_\mathcal{O}}}\frac{\|u\|^2_{H^2_*}}{\|u\|^2_{2}}=\frac{\nu_{1}(1)}{\beta};$$
%hence, if $\frac{\mu_{1}(1)}{\alpha}<\frac{\nu_{1}(1)}{\beta}$ we obtain $\mu_1(p)<\nu_1(p)$.
%
\subsection{Proof of Theorem \ref{existence}}
The proof follows by combining three lemmas that we state here below. In the first part of this section we will not need to distinguish between longitudinal and torsional eigenvalues.\par  
Given $p\in P_{\alpha, \beta}$, it is convenient to endow the space $L^2$ of the weighted scalar product: $(p\, u,v)_{L^2}$, for all $u,v\in L^2$, which defines an equivalent norm in $L^2$. Then, for $h\in N_+$, we introduce the orthogonal projection of $u \in H^2_*$, with respect to the above weighted scalar product, onto the space generated by the first $(h-1)$ eigenfunctions $u_1,\dots, u_{h-1}$ of problem \eqref{weight}: 
 $$P_{h-1}(p)u:=\sum_{i=1}^{h-1}(p\, u,u_i)_{L^2}\, u_i\,;$$ 
when $h=1$ we adopt the convection $P_0(p)u =0$. Finally, we recall the Auchmuty's principle \cite{auchmuty} stated in our framework:
\begin{lemma}\label{lemmaauchmuty}
	Let $p\in P_{\alpha, \beta}$ and $\lambda_h(p)$ the $h-$th eigenvalue of \eqref{weight} with $h\in N_+$, then 
	$$
	-\dfrac{1}{2\lambda_h(p)} =\inf_{u \in H^2_*}\mathcal{A}_h(p,u)\quad \text{where}\quad \mathcal{A}_h(p,u):=\dfrac{1}{2}\|u\|^2_{H^2_*}-\|\sqrt{p}\,\big[u-P_{h-1}(p)u\big]\|_{2}\,.
	$$
 Furthermore, the minimum is achieved at a $h$-th eigenfunction normalized according to
	$$\|u_h\|^2_{H^2_*}=\|\sqrt{p}\,u_h\|_{2}=\dfrac{1}{\lambda_h(p)}$$
\end{lemma}
\begin{proof}
	 The proof follows arguing as in \cite[Lemma 3.3]{cuccu22} by simply replacing $H^2\cap H^1_0$ with $H^2_*$. In alternative, in \cite{auchmuty} one can find the original proof in a general setting.
\end{proof}
\begin{lemma}\label{compactness}
	 The set $P_{\alpha,\beta}$ is compact for the weak* topology of $L^\infty$.
\end{lemma}
\begin{proof}
	First we prove that $P_{\alpha,\beta}$ is a strongly closed set in $L^2$.\par  
	Let $\{ p_m \}_m \subset P_{\alpha,\beta}$ be a sequence such that $p_m\rightarrow q$ in $L^2$ (as $m\rightarrow +\infty$) for some $q\in L^2$; then $p_m\rightarrow q$ in $L^1$ (as $m\rightarrow +\infty$) and up to a subsequence (still denoted by $p_m$) we infer that $p_m \rightarrow q$ a.e. in $\Omega$. Therefore, $\alpha\leq q\leq\beta$ and $q$ is $y$-even a.e. in $\Omega$; moreover, $
	\int_{\Omega}p_m\, v\,dx\,dy\rightarrow \int_{\Omega}q \,v\,dx\,dy$ for all $v\in L^2,
	$ so that, choosing $v\equiv 1\in L^2$, we obtain $|\Omega|=\int_{\Omega}q\,dx\,dy$. This implies that $q\in P_{\alpha,\beta}$ and $P_{\alpha,\beta}$ is strongly closed in $L^2$.  \par  
	Next, we show that any sequence $\{ p_m \}_m \subset P_{\alpha,\beta}$ admits a subsequence converging in the weak* topology of $L^\infty$ to an element of  $P_{\alpha,\beta}$.  By the definition of $P_{\alpha,\beta}$  we have $\|p_m\|_\infty\leq \beta$, so that, up to a subsequence, we obtain
$$
p_{m_k} \overset{\ast}{\rightharpoonup}  \overline{p} \quad \mbox{in} \ L^\infty\quad \mbox{as} \ k \rightarrow \infty\,.
$$
Moreover, we have $\|p_{m_k}\|_2^2\leq \beta|\Omega|$, so that, up to a subsequence, we infer that 
$p_{m_{k_{j}}}\rightharpoonup \overline q$ in $L^2$ as $j\rightarrow \infty$.
It is easy to check that $P_{\alpha,\beta}$ is a convex set and, since convex strongly closed space are weakly closed, we readily infer that $\overline q\in P_{\alpha,\beta}$.\par Therefore,
	$$
	 \int_{\Omega}p_{m_{k_{j}}}\, v\,dx\,dy\rightarrow\int_{\Omega}\overline q\, v\,dx\,dy \quad \forall v\in L^2\subset L^1\quad \mbox{as} \ j \rightarrow \infty \quad \text{with }\overline q\in P_{\alpha,\beta}\,
	$$
	and, since $p_{m_k} \overset{\ast}{\rightharpoonup}  \overline{p} \quad \mbox{in} \ L^\infty$ yields $\int_{\Omega}p_{m_{k_{j}}}\, v\,dx\,dy\rightarrow\int_{\Omega}\overline p\, v\,dx\,dy$ $\forall v\in L^1$, we conclude that $\overline p=\overline q$ a.e. in $\Omega$. Whence, $\overline p\in P_{\alpha,\beta}$ and the proof is complete.
\end{proof}
\begin{lemma}\label{continuity}
	Let $\lambda_h(p)$ the $h-$th eigenvalue of \eqref{weight} with $h\in N_+$. The map $p\mapsto\lambda_h(p)$ is continuous on $P_{\alpha,\beta}$ for the weak* convergence.
\end{lemma}
\begin{proof}
	Let $\{ p_m \}_m \subset P_{\alpha,\beta}$ be a sequence converging in the weak* topology of $L^\infty$ to $\overline{p}$, i.e.
	$$
	p_m \overset{\ast}{\rightharpoonup}  \overline{p} \quad \mbox{in} \ L^\infty \quad \mbox{as} \ m \rightarrow \infty;
	$$
then $\overline{p}\in P_{\alpha,\beta}$ by Lemma \ref{compactness}.\par 
To $p_m$ we associate the $h$-th eigenvalue  $\lambda_h(p_m)$ of \eqref{weight} and an eigenfunction $u_h(p_m)$ normalized with respect to the weighted scalar product, i.e. $\int_\Omega p_m\,u_h(p_m)\,u_r(p_m)\,dx\,dy=\delta_{hr}$, where $\delta_{hr}$ is the Kronecker delta for all $h,r \in \mathbb{N}_+$ and $\lambda_h(p_m)=\|u_h(p_m)\|^2_{H^2_*}$.\par 
By \eqref{eq:famiglia} and \eqref{caract1} we have
$$
\lambda_h(p)\leq \dfrac{\lambda_h(1)}{\alpha}\qquad \forall p\in P_{\alpha,\beta},
$$
where $\lambda_h(1)$ is the $h$-th eigenvalue of \eqref{weight} with $p\equiv 1$,
implying that
$\lambda_h(p_m)=\|u_h(p_m)\|_{H^2_*} \leq \lambda_h(1)/\alpha$. Therefore, we can extract a subsequence, still denoted
by $u_h(p_m),$ such that
\begin{equation*}\label{wconvergence}
\begin{split}
&\lambda_h(p_m)\rightarrow  \overline \lambda_h\quad \mbox{in} \ \mathbb{R},\\
&u_h(p_m) \rightharpoonup  \overline{u}_h \quad \mbox{in} \ H^2_* \quad \mbox{as} \ m \rightarrow \infty.
\end{split}
\end{equation*}
Moreover, due to the compact embedding $H^2_* \hookrightarrow L^2$, we obtain that $
u_h(p_m)$ strongly converges to $\overline{u}_h$ in $L^2$ as $m \rightarrow \infty$; this implies, for all $v\in H^2_*$, that 
\begin{equation*}\label{pconv}
\int_\Omega p_m\,u_h(p_m)\,v\,dx\,dy\rightarrow\int_\Omega \overline p\,\overline u_h\,v\,dx\,dy \quad \mbox{as} \ m \rightarrow \infty
\end{equation*}
indeed
\begin{equation*}
\bigg|\int_\Omega (p_m\,u_h(p_m)-\overline p\,\overline u_h)\,v\,dx\,dy\bigg|\leq \|p_m v\|_2\|u_h(p_m)-\overline u_h\|_2+\bigg|\int_\Omega p_m\,\overline u_hv\,dx\,dy-\int_{\Omega} \overline p\,\overline u_hv\,dx\,dy\bigg|\rightarrow 0,
\end{equation*}
since $\overline u_h v\in H^2_*\subset L^1$. Therefore, we obtain
\begin{align*}
\big(u_h(p_m),v\big)_{H^2_*}-\lambda_h(p_m)\big( p_m\,u_h(p_m),v\big)_{L^2}\rightarrow \big(\overline u_h,v\big)_{H^2_*}-\overline \lambda_h\big( \overline p\,\overline u_h,v\big)_{L^2}\qquad \forall v\in H^2_* \quad \mbox{as} \ m \rightarrow \infty\,,
\end{align*}
 inferring that $\overline\lambda_h$ is an eigenvalue of \eqref{weight} and $\overline u_h$ is a corresponding eigenfunction. \par
Arguing as before we also obtain
$\int_\Omega p_m\,u_h(p_m)\,u_r(p_m)\,dx\,dy\rightarrow\int_\Omega \overline p\,\overline u_h\,\overline u_r\,dx\,dy=\delta_{hr}$ for all $h,r \in\mathbb{N}_+$, so that
$\overline \lambda_h$ is a diverging sequence for $h\rightarrow\infty$. To prove that $\overline \lambda_h=\lambda_h(\overline p)$ for every $h\in \mathbb{N}_+$, we assume by contradiction that, for $p=\overline p$, there exists an eigenfunction $\overline u$ associated with the eigenvalue $\overline \lambda$ such that $\big(\overline p\, \overline u,\overline u_h \big)_{L^2}=0$ for all $h\in\mathbb{N}_+$.
We suppose that $\overline u$ is normalized in $H^2_*$ so that $\|\sqrt{\overline p}\,\overline u\|_{2}=1/\overline \lambda$; applying Lemma \ref{lemmaauchmuty} we have 
\begin{equation}\label{assurdo}
-\dfrac{1}{2\lambda_h(p_m)}\leq \mathcal{A}_h(p_m,\overline u)=\dfrac{1}{2}\|\overline u\|^2_{H^2_*}-\|\sqrt{p_m}\,\big[\overline u-P_{h-1}(p_m)\overline u\big]\|_{2}\rightarrow \dfrac{1}{2}\|\overline u\|^2_{H^2_*}-\|\sqrt{\overline p}\,\overline u\|_{2}=-\dfrac{1}{2\overline \lambda},
\end{equation}
where the convergence comes from
$$
P_{h-1}(p_m)\overline u=\sum_{i=1}^{h-1}\big( p_m\overline u, u_i(p_m) \big)_{L^2}\,u_i(p_m)\rightarrow \sum_{i=1}^{h-1}\big(\overline p\, \overline u, \overline u_i \big)_{L^2}\,\overline u_i=0\quad {\rm in} \hspace{1mm} L^2.
$$
Therefore, by \eqref{assurdo}, letting $m\rightarrow \infty$, we obtain
$$
\overline \lambda\geq \lambda_h(p_m)\rightarrow \overline \lambda_h\qquad \forall h\in\mathbb{N}_+,
$$
giving a contradiction since $\overline \lambda_h$ is an unbounded sequence for $h\rightarrow\infty$. Thus $\overline \lambda_h= \lambda_h(\overline p)$, implying the continuity of $p\mapsto\lambda_h(p)$ for every $h\in\mathbb{N}_+$ fixed.
\end{proof} 
{\bf Proof of Theorem \ref{existence} completed.}
Let us consider the function $F: (0,+\infty)\times (0,+\infty)\mapsto \mathbb{R}$ given by $F(t,s):=\dfrac{t}{s}$, continuous on its domain.  By Lemma \ref{continuity}, the maps $p\mapsto \nu_1(p)$ and $p\mapsto \mu_{j_0}(p)$ are continuous on $P_{\alpha,\beta}$ for the weak* convergence; since $\nu_1(p)>\mu_{j_0}(p)>0$, we infer that $F(\mu_{j_0}(p),\nu_1(p))=\dfrac{\nu_1(p)}{\mu_{j_0}(p)}$ is also continuous on $P_{\alpha,\beta}$ for the weak* convergence. Finally, the existence of a maximum (or minimum) of $F(\mu_{j_0}(p),\nu_1(p))=\dfrac{\nu_1(p)}{\mu_{j_0}(p)}$ on $P_{\alpha,\beta}$ follows thanks to the compactness proved in Lemma \ref{compactness} of the set $P_{\alpha,\beta}$ for the weak* topology of $L^{\infty}$.\par

\subsection{Proof of Theorem \ref{thm-nu}}
The existence of an optimal pair for \eqref{CP2} follows as in the proof of Theorem \ref{existence} by considering the continuous function $F(\mu_{j_0}(p),\nu_1(p))=\nu_1(p)$. In the sequel we will denote by $(\widehat{p},\widehat u) \in P_{\alpha, \beta}\times H^2_\mathcal{O}$ an optimal pair for \eqref{CP2} suitable normalized as follows
\begin{equation}\label{norm}
\nu^{\alpha,\beta}_1= \, \nu_1(\widehat{p})= \frac{\|\widehat  u\|_{H^2_*}^2}{\|\sqrt{\widehat{p}}\,\,\widehat  u\|_{2}^2}\quad \text{and} \quad \|\widehat u \|^2_{H^2_*}=\dfrac{1}{\nu_1(\widehat{p})}\,.
\end{equation}

Next we state a couple of lemmas useful to complete the proof.
\begin{lemma}\label{rearrangement}
Let $u\in H^2_* \setminus \{0\}$ and let $J:P_{\alpha,\beta} \rightarrow \R$ be defined as follows
$$J(p)=\int_{\Omega} p(x,y)u^2\,dx\,dy\,.$$
Then, the problem
$$I_{\alpha, \beta}:=\inf_{p\in P_{\alpha,\beta} } J(p)$$
admits the solution 
$$
{p}_u(x,y) = \beta \chi_{\widetilde{ S}} (x,y)+ \alpha \chi_{\Omega \setminus \widetilde{{S}}}(x,y)\,\quad \text{for a.e. } (x,y)\in \Omega\,,
$$
	where $ \widetilde{S}=\widetilde{S}(u) \subset \Omega$
	is such that $|\widetilde{S}|=\frac{1-\alpha}{\beta-\alpha}\,|\Omega|:=\widetilde C_{\alpha, \beta}$. Furthermore, define
	\begin{equation} \label{eq:def-t}
	t:= \sup \left\{ s \geq 0 : |\{(x,y)\in \Omega\,: u^2(x,y) \leq s \}| < \widetilde C_{\alpha, \beta}\right\},
	\end{equation}
	if $t=0$ we have that
	\begin{equation*}
	%\label{S1}
	\widetilde{S}\subseteq\{(x,y)\in \Omega \, :u^2(x,y)=0 \} 
	\end{equation*}
	while if $t>0$ we have that
	\begin{equation}\label{S}
	 \{(x,y)\in \Omega\,:u^2(x,y) < t \} \subseteq \widetilde{S} \subseteq \{(x,y)\in \Omega\,: u^2(x,y)\leq t\}\,.
	\end{equation}
	
	Similarly, the problem
	$$M_{\alpha, \beta}:=\sup_{p\in P_{\alpha,\beta} } J(p)$$
	admits the solution 
	$$
	p_u(x,y) = \alpha \chi_{\check{ S}} (x,y)+ \beta \chi_{\Omega \setminus \check{{S}}}(x,y)\,\quad \text{for a.e. }(x,y)\in \Omega\,,
	$$
	where $ \check{S}=\check{S}(u) \subset \Omega$
	is such that $|\check{S}|=\frac{\beta-1}{\beta-\alpha}\,|\Omega|=:\check C_{\alpha, \beta}$ and satisfies \eqref{S}-\eqref{eq:def-t} with $\widetilde C_{\alpha, \beta}$ replaced by $\check C_{\alpha, \beta}$.	

\end{lemma}
\begin{proof}
We only prove the statement for $I_{\alpha, \beta}$ since the statement for $M_{\alpha, \beta}$ follows basically by reversing all the inequalities below. Since ${p}_u\in P_{\alpha,\beta}$ we have
$$
I_{\alpha,\beta}\leq J({ p}_u)\,.
$$
If we prove that 
$$
J(p)\geq J({ p}_u)\qquad \forall p\in P_{\alpha,\beta},
$$
the thesis is obtained. To this aim, we first consider the case $t>0$, we have
\begin{align*} \label{eq:optimal2}
& \int_{\Omega} u^2 (p_u-p) \, dx\, dy \\[8pt]
\notag &  =\int_{\{ u^2 < t\}} u^{2} \, (\beta-p) \, dx\, dy + \int_{\{u^2 > t\}} u^{2} \, (\alpha-p) \, dx\, dy +
\int_{\{ u^2 = t\}} u^{2} \, (p_u-p) \, dx\, dy \\[8pt]
\notag & \leq t \int_{\{ u^2 < t\}} (\beta-p) \, dx\, dy + t \int_{\{ u^2 > t\}} (\alpha-p) \, dx\, dy
+t \int_{\{u^2 = t\}} \, (p_u-p) \, dx\, dy \\[8pt]
& \notag  = t \int_{\Omega} (p_u-p) \, dx\, dy=0 \,,
\end{align*}
where the last equality comes from the preservation of the total mass condition. \par

Similarly, for $t=0$ we have
$$
\int_{\Omega} u^2 (p_u-p) \, dx\, dy = \int_{\{ u^2 > 0\}} u^2\,(p_u-p)\, dx\, dy\leq 0.
$$
In both the cases we conclude that $J(p)\geq J({ p}_u)$, and in turn that $J({ p}_u)=I_{\alpha, \beta}$.\par

\end{proof}

We will also invoke the Auchmuty's principle recalled in Lemma \ref{lemmaauchmuty} that, in terms of $\nu_1(p)$, rewrites
	\begin{equation}\label{AP}
	-\dfrac{1}{2\nu_1(p)} =\inf_{u \in H^2_\mathcal{O}}\mathcal{A}(p,u)\qquad \mathcal{A}(p,u):=\dfrac{1}{2}\|u\|^2_{H^2_*}-\|\sqrt{p}\,u\|_{2}\,.
	\end{equation}
	By this,  if $(\widehat{p},\widehat u)\in P_{\alpha,\beta}\times H^2_\mathcal{O} $ and \eqref{norm} is satisfied, it is readily deduced that
	\begin{equation}\label{AP2}
	\sup_{p\in P_{\alpha, \beta}} \inf_{u \in H^2_\mathcal{O}}\mathcal{A}(p,u) = -\dfrac{1}{2\nu^{\alpha,\beta}_1}=-\dfrac{1}{2\nu_1(\widehat{p})} =\inf_{u \in H^2_\mathcal{O}}\mathcal{A}(\widehat{p},u) =\mathcal{A}(\widehat{p},\widehat u)\,.
	\end{equation}
Furthermore, we have
\begin{lemma} \label{infsup}
	Let $\mathcal{A}(p,u)$ be as defined in \eqref{AP}, the following equality holds
	\begin{equation}\label{scambio}
\sup_{p\in P_{\alpha, \beta}} \inf_{u \in H^2_\mathcal{O}}\mathcal{A}(p,u) =\inf_{u \in H^2_\mathcal{O}} \sup_{p\in P_{\alpha, \beta}} \mathcal{A}(p,u)\,.
\end{equation}
\end{lemma}
The proof of Lemma \ref{infsup} is the same of \cite[Lemma 3.7]{cuccu22} once replaced the set $H^2\cap H^1_0$ there with our set $H^2_\mathcal{O}$ (strongly and weakly closed subspace of $H^2$).  Hence, we omit it and we refer the interested readers to \cite{cuccu22} or \cite{cox}, where the proof was originally given in the second order case. \par
Finally, we prove

\begin{lemma}\label{sella}
There exists an optimal pair $(\widehat{p},\widehat u)\in P_{\alpha,\beta}\times H^2_\mathcal{O} $ as in \eqref{norm} such that there holds
\begin{equation}\label{crucial}
\int_{\Omega} \widehat{p}(x,y)\, \widehat u^2\, dx dy\leq \int_{\Omega} p(x,y)\, \widehat u^2\, dx dy \qquad \forall p \in P_{\alpha, \beta}\,.
\end{equation}
\end{lemma}

\begin{proof}
The idea of the proof is taken from \cite[Proposition 3.8]{cuccu22}, the main difference here is the use of Lemma \ref{rearrangement}. \par
First we consider the functional $B: H^2_\mathcal{O} \rightarrow \R$ defined as follows
$$B(u)= \sup_{p\in P_{\alpha, \beta}} \mathcal{A}(p,u)=\dfrac{1}{2}\|u\|^2_{H^2_*}-\inf_{p\in P_{\alpha, \beta}}\|\sqrt{p}\,u\|_{2} \,,$$
and we claim that
\begin{equation}\label{claim}
\text{there exists } \bar u \in H^2_\mathcal{O} \text{ such that } B(\bar u)= \inf_{u \in H^2_\mathcal{O}} B(u)=\inf_{u \in H^2_\mathcal{O}}  \sup_{p\in P_{\alpha, \beta}} \mathcal{A}(p,u)=:I\,.
\end{equation}
To this aim, let $\{u_k\}$ be a minimizing sequence for $I$, namely
$$\dfrac{1}{2}\|u_k\|^2_{H^2_*}-\inf_{p\in P_{\alpha, \beta}}\|\sqrt{p}\,u_k\|_{2}=I+o(1) \quad \text{as } k\rightarrow + \infty\,.$$
By the boundedness of $p$ and the continuity of the embedding $H^2_*\subset L^2$ it is readily deduced that
$$\|u_k\|^2_{H^2_*}\leq C \|u_k\|_{H^2_*}+2I+o(1) \quad \text{as } k\rightarrow + \infty$$
and, in turn, that 
$$\|u_k\|_{H^2_*}\leq \overline C  \quad \text{for } k\text{ sufficiently large,}$$
with $C, \overline C>0$. Then, up to a subsequence, we have
$$u_k \rightharpoonup \overline u \text{ in } H^2_*\quad  \text{ and } \quad u_k \rightarrow \overline u \text{ in } L^2\quad \text{as } k\rightarrow + \infty\,.$$
Therefore,
$$\|\overline u\|^2_{H^2_*}\leq \liminf_{k\rightarrow + \infty}\|u_k\|^2_{H^2_*}\,.$$
Next we take ${p}_{\overline u}$ as given in Lemma \ref{rearrangement}, namely such that
$$\inf_{p\in P_{\alpha,\beta} } \int_{\Omega} p(x,y)\, \overline u^2\,dx\,dy= \int_{\Omega} {p}_{\overline u}(x,y) \, \overline u^2\,dx\,dy\,.$$
Clearly,
$$\lim_{k\rightarrow + \infty} \int_{\Omega} {p}_{\overline u}(x,y) \, u_k^2\,dx\,dy=\int_{\Omega} {p}_{\overline u}(x,y) \, \overline u^2\,dx\,dy$$
and 
$$\inf_{p\in P_{\alpha,\beta} } \int_{\Omega} p(x,y)\, u_k^2\,dx\,dy\leq  \int_{\Omega} {p}_{\overline u}(x,y) \, u_k^2\,dx\,dy\,.$$
In particular, we conclude that
$$\limsup_{k\rightarrow + \infty}\inf_{p\in P_{\alpha, \beta}} \int_{\Omega} p(x,y)\, u_k^2\,dx\,dy\leq \limsup_{k\rightarrow + \infty} \int_{\Omega} {p}_{\overline u}(x,y) \,u_k^2\,dx\,dy=\inf_{p\in P_{\alpha,\beta} } \int_{\Omega} p(x,y)\, \overline u^2\,dx\,dy\,.$$
The above inequalities yield
$$B(\overline u)\leq \liminf_{k\rightarrow + \infty} B(u_k)=I$$
which is the claim \eqref{claim}. \par \smallskip\par

By combining  \eqref{AP2}, \eqref{scambio} and \eqref{claim}, it follows that
$$ \sup_{p\in P_{\alpha, \beta}} \mathcal{A}(p,\overline u)=\inf_{u \in H^2_\mathcal{O}}  \sup_{p\in P_{\alpha, \beta}} \mathcal{A}(p,u)=\sup_{p\in P_{\alpha, \beta}}\inf_{u \in H^2_\mathcal{O}}   \mathcal{A}(p,u)=\mathcal{A}(\widehat p, \widehat u)= \inf_{u \in H^2_\mathcal{O}}   \mathcal{A}(\widehat p,u)\,.$$
In particular, this implies that
\begin{equation}\label{ineqbaru}
\mathcal{A}(p,\overline u)\leq  \sup_{p\in P_{\alpha, \beta}} \mathcal{A}(p,\overline u) =  \inf_{u \in H^2_\mathcal{O}}   \mathcal{A}(\widehat p,u)\leq  \mathcal{A}(\widehat p,\overline u) \qquad \forall p\in P_{\alpha, \beta} 
\end{equation}
and that
$$\mathcal{A}(\widehat p,\overline u)\leq  \sup_{p\in P_{\alpha, \beta}} \mathcal{A}(p,\overline u)=  \inf_{u \in H^2_\mathcal{O}}   \mathcal{A}(\widehat p,u)\leq  \mathcal{A}(\widehat p, u) \qquad \forall u\in  H^2_\mathcal{O}\,.$$
From the above inequality we infer that $\overline u$ is a minimizer of $\mathcal{A}(\widehat p, u)$, hence an eigenfunction of $\nu_1(\widehat p)$, see \eqref{AP2}. Then, $(\widehat p,\overline u)$ is an optimal pair as defined in \eqref{norm} and we may take $\widehat u=\overline u$ in the statement. Furthermore, by \eqref{ineqbaru} with $\widehat u=\overline u$ we get
$$\mathcal{A}(p,\widehat u)\leq   \mathcal{A}(\widehat p,\widehat u) \qquad \forall p\in P_{\alpha, \beta} \,,$$
which, recalling the definition of $\mathcal{A}( p, u) $, yields \eqref{crucial}.

\end{proof}
 	
{\bf Proof of Theorem \ref{thm-nu} completed.}
\par \smallskip\par

\textbf{ Step 1.} 
%We prove that $(p_{\widehat  u},\widehat u) \in P_{\alpha,\beta} \times H^2_\mathcal{O} $, with $p_{\widehat  u}:=\beta \chi_{S_{\widehat u}} + \alpha \chi_{\Omega \setminus S_{\widehat u} }$ and  $S_{\widehat u}=\widetilde{S}({\widehat u})$ as defined in Lemma \ref{rearrangement}, is still an optimal pair. 
Let $(\widehat{p},\widehat u)\in P_{\alpha,\beta}\times H^2_\mathcal{O} $ be the optimal pair given by Lemma \ref{sella} and set $p_{\widehat  u}(x,y):=\beta \chi_{S_{\widehat u}}(x,y) + \alpha \chi_{\Omega \setminus S_{\widehat u} }(x,y)$ with $S_{\widehat u}=\widetilde{S}({\widehat u})$ as defined in Lemma \ref{rearrangement}; we prove that
\begin{equation}\label{equality}
\|\sqrt{\widehat{p}}\,\,\widehat  u\|_2^2=\|\sqrt{p_{\widehat  u}}\,\,\widehat  u\|_2^2.
\end{equation}
\par \smallskip\par
By Lemma \ref {rearrangement} on problem $I_{\alpha,\beta}$ we know that
$$\int_{\Omega} \widehat p(x,y)\, \widehat u^2\,dx\,dy\geq \int_{\Omega} p_{\widehat  u}(x,y)\, \widehat u^2\,dx\,dy\,.$$
On the other hand, by \eqref{crucial} with $p=p_{\widehat  u}$ we infer
$$\int_{\Omega} \widehat{p}(x,y)\, \widehat u^2\, dx dy\leq \int_{\Omega} p_{\widehat  u}(x,y)\, \widehat u^2\, dx dy \,.$$
Comparing the above inequalities, the proof of Step 1 follows.
\par
\bigskip\par
\textbf{ Step 2.}  Let $(p_{\widehat  u},\widehat u) \in P_{\alpha,\beta}\times H^2_\mathcal{O}$ be as in Step 1 and let $\widehat t\geq 0$ be the number corresponding to $\widehat u$ in $S_{\widehat u}$. We prove that $\widehat p=p_{\widehat u}$ a.e. in $\Omega$.
\par \smallskip\par
 By \eqref{equality}, if $\widehat t=0$ we have 
$$
0=\int_{\Omega}(\widehat p-p_{\widehat u}) \widehat u^2\,dx\,dy=\int_{\{\widehat u^2=0\}}(\widehat p-p_{\widehat u}) \widehat u^2\,dx\,dy+\int_{\{\widehat u^2>0\}}(\widehat p-\alpha) \widehat u^2\,dx\,dy,
$$
implying  $\widehat p\, \widehat u=\alpha \, \widehat u$ a.e. in $\Omega$.
On the other hand, since $\widehat u \in H^4(\Omega)$ we can write almost everywhere the Euler-Lagrange equation related to the Rayleigh quotient of $\nu_1^{\alpha,\beta}=\nu_1(\widehat p)$ and, for what observed above, we infer that
$$
\Delta^2 \widehat u= \nu_1^{\alpha,\beta}\,\alpha \, \widehat u \quad \text{a.e. in } \Omega \, . $$
Recalling that $\widehat u$ satisfies the partially hinged boundary conditions, this means that it must be one of the eigenfunctions listed in Proposition \ref{eigenvalue}; since the set of zeroes of any of the eigenfunctions of Proposition \ref{eigenvalue} has zero measure, this contradicts the definition
of $S_{\widehat u}$ and forces $S_{\widehat u}$ to be a set of positive measure. Whence, the above arguments proves that $\widehat t>0$.\par For $\widehat t>0$ we set
$$
A_{\widehat t}=\{ (x,y)\in \Omega\,: \widehat u^2(x,y) = \widehat t\,\} \,.
$$
By \eqref{equality} we obtain
$$
0=\int_{\Omega}(\widehat p-p_{\widehat u}) \widehat u^2\,dx\,dy=\int_{\{\widehat u^2<\widehat t\}}(\widehat p-\beta) \widehat u^2\,dx\,dy+\int_{\{\widehat u^2>\widehat t\}}(\widehat p-\alpha) \widehat u^2\,dx\,dy+\int_{A_{\widehat t}}(\widehat p-p_{\widehat u}) \widehat u^2\,dx\,dy.
$$
Assume by contradiction that $\widehat p>\alpha$ in a set $A\subseteq \{\widehat u^2>\widehat t\}$ with $|A|>0$, then we get
$$
\int_{\{\widehat u^2>\widehat t\}}(\widehat p-\alpha)( \widehat u^2-\widehat t)\,dx\,dy\geq \int_{A}(\widehat p-\alpha) ( \widehat u^2-\widehat t)\,dx\,dy>0\,
$$
and, in turn, that $\int_{\{\widehat u^2>\widehat t\}}(\widehat p-\alpha) \widehat u^2\,dx\,dy>\widehat t\,\int_{\{\widehat u^2>\widehat t\}}(\widehat p-\alpha)\,dx\,dy$. Whence,
$$
0>\int_{\{\widehat u^2<\widehat t\}}(\widehat p-\beta) \widehat u^2\,dx\,dy+\widehat t\int_{\{\widehat u^2>\widehat t\}}(\widehat p-\alpha)\,dx\,dy +\widehat t\int_{A_{\widehat t}}(\widehat p-p_{\widehat u})\,dx\,dy \geq \widehat t\int_{\Omega}(\widehat p-p_{\widehat u})\,dx\,dy=0\,,
$$
where the last equality follows from the preservation of the total mass condition. This contradicts the definition of the set $A$ and implies $\widehat p=\alpha$ a.e. in $\{\widehat u^2>\widehat t\}$. Proceeding as before, we suppose that $\widehat p<\beta$ in a subset of positive measure of $\{\widehat u^2<\widehat t\}$ and we obtain a further contradiction
$$
0>\widehat t\int_{\{\widehat u^2<\widehat t\}}(\widehat p-\beta) \,dx\,dy+\widehat t\int_{A_{\widehat t}}(\widehat p-p_{\widehat u})\,dx\,dy=\widehat t\int_{\Omega}(\widehat p-p_{\widehat u})\,dx\,dy=0.
$$
It remains to study $\widehat p$ in $A_{\widehat t}$. When $|A_{\widehat t}|>0$, we write the Euler-Lagrange equation related to $\nu_1^{\alpha,\beta}=\nu_1(\widehat p)$ obtaining
$$
\Delta^2 \widehat u= \nu_1^{\alpha,\beta}\,\widehat p \, \widehat u \quad \text{a.e. in } A_{\widehat t}. \ $$
Since $\widehat u^2=\widehat t$ we get $0= \nu_1^{\alpha, \beta} \, \widehat p$, that is absurd since $\widehat p\geq \alpha>0$. This implies that $A_{\widehat t}$ must have zero measure, so that $\widehat p=p_{\widehat u}$ a.e. in $\Omega$. 
\bigskip\par
\textbf{ Step 3.} 
Since $|A_{\widehat{t}}|=0$, also $|A_{\widehat t}\setminus S_{\widehat u}|=0$,
therefore it is not restrictive, up to a set of zero measure, to assume that
$A_{\widehat t}\setminus S_{\widehat u}=\emptyset$ in such way that $A_{\widehat t}\subseteq S_{\widehat u}$ and, in turn, that
\begin{equation*} \label{eq:Su2}
S_{\widehat u}=\{(x,y)\in \Omega: \widehat u^2(x,y)\leq \widehat t\} \, .
\end{equation*}
\par

\subsection{Proof of Theorem \ref{thm-muj}}
The existence issue follows as in the proof of Theorem \ref{existence} by considering the continuous function $F(\mu_{j}(p),\nu_1(p))=\mu_j(p)$.\par 
Fixed $j\geq 2$, by the characterization \eqref{caract1} of $\mu_j(p)$, we may choose $\overline W_j=\{w_1,\dots,w_j\}\subset H^2_\mathcal{E}$, where the functions $w_m$ for $m=1,\dots,j$ are defined in \eqref{sin2}. Therefore, we obtain
\begin{equation*}\label{stima4}
\mu_j(p)\leq \sup_{\substack{u\in \overline W_j\setminus \{0\}}} \frac{\|u\|^2_{H^2_*}}{\|\sqrt{p}u\|^2_2}= \max_{\substack{m=1,\dots,j}}\bigg\{\frac{\|w_{m}\|^2_{H^2_*}}{\|\sqrt{p} \,w_{m}\|^2_2}\bigg\}\,,
\end{equation*}
where we have exploited the fact that the space $\overline W_j$ is generated by disjointly supported functions. Since $\|w_m\|^2_{H^2_*}=|\Omega|j^3$, we finally obtain the upper bound \eqref{stimap}.\par 
In order to reduce this bound and, in turn $\mu_j^{\alpha,\beta}$, it is convenient to choose a weight $p$ having the same effect on every $\|\sqrt{p}\,w_m\|_2^2$; this suggests to take a weight $\pi/j$-periodic in $x$, i.e. $p\in P_{\alpha,\beta}^{per}$.
In this way we obtain
$$
\|\sqrt{p} w_m\|^2_2=\dfrac{1}{j}\|\sqrt{p} \sin^2(jx)\|^2_2,$$
and the proof of \eqref{stimap1} readily follows from \eqref{stimap}.\par
The last part of the statement comes out by slightly modifying the proof of Lemma \ref{rearrangement} by which we infer that the problem
	$$\inf_{p\in P_{\alpha,\beta}^{per} } \int_{\Omega}p(x,y)\,\sin^4(jx)\,dx \,dy$$
	admits the solution 
	\begin{equation*}\label{pm}
	p_j(x,y) = \alpha \chi_{ S_j} (x,y)+ \beta \chi_{\Omega \setminus S_j}(x,y)\,\quad \text{for a.e. } (x,y)\in \Omega\,,
	\end{equation*}
	where $ S_j = \{ (x,y)\in \Omega\,:\, \sin^4(jx)  \leq t_j \}$ for $ t_j> 0$  such that $| S_j|=\frac{\beta-1}{\beta-\alpha}\,|\Omega|$.
% Indeed we know that the set on which $\sin^4(jx)=t_j$ has zero measure for every $t_j\geq 0$. Therefore, up to a set of zero measure, we may take $S_j=\check{S}(u_j) = \{(x,y)\in \Omega\,: \sin^4(jx)\leq t_j\} $.
\section{Appendix}
%\label{appendix}
 The aim of this section is to give further information on the eigenvalues  $\lambda_h(p)$ $(h\in \N_+)$ of \eqref{weight}; in particular, we compare problem \eqref{weight} with the more studied Dirichlet and Neumann type problems and we derive a Weyl-type asymptotic law for $\lambda_h(p)$ . \par  We begin by writing the above mentioned problems in our rectangular domain $\Omega$; the Dirichlet problem reads
\begin{equation}\label{dir}
\begin{cases}
\Delta^2 u=\lambda\, p(x,y) u & \qquad \text{in } \Omega \\
u(0,y)=u_{x}(0,y)=u(\pi,y)=u_{x}(\pi,y)=0 & \qquad \text{for } y\in (-\ell,\ell)\\
u(x,\pm\ell)=u_{y}(x,\pm\ell)=0
& \qquad \text{for } x\in (0,\pi)\, ,
\end{cases}
\end{equation} 
 with weak form
 \begin{equation*}\label{diri}
	\int_\Omega \Delta u\Delta v=\lambda \int_\Omega puv\qquad\forall v\in H^2_0.
	\end{equation*} 
%	The Navier problem reads
%\begin{equation}\label{nav}
%\begin{cases}
%\Delta^2 u=\lambda\, p(x,y) u & \qquad \text{in } \Omega \\
%u(0,y)=u_{xx}(0,y)=u(\pi,y)=u_{xx}(\pi,y)=0 & \qquad \text{for } y\in (-\ell,\ell)\\
%u(x,\pm \ell)=u_{yy}(x,\pm \ell)=0
%& \qquad \text{for } x\in (0,\pi)\, ,
%\end{cases}
%\end{equation} 
%with weak form 
%$$\int_\Omega\Delta u\Delta v\, dx \, dy\,=\lambda \int_\Omega p(x,y)\,uv\, dx \, dy\,\qquad\forall v\in H^2\cap H^1_0\,.$$
The Neumann type problem reads
\begin{equation}\label{neu}
\begin{cases}
\Delta^2 u=\lambda\, p(x,y) u & \qquad \text{in } \Omega \\
u_{xx}(0,y)+\sigma
u_{yy}(0,y)=u_{xxx}(0,y)+(2-\sigma)u_{yyx}(0,y)=0 & \qquad \text{for } y\in (-\ell,\ell)\\
u_{xx}(\pi,y)+\sigma
u_{yy}(\pi,y)=u_{xxx}(\pi,y)+(2-\sigma)u_{yyx}(\pi,y)=0 & \qquad \text{for } y\in (-\ell,\ell)\\
u_{yy}(x,\pm\ell)+\sigma
u_{xx}(x,\pm\ell)=u_{yyy}(x,\pm\ell)+(2-\sigma)u_{xxy}(x,\pm\ell)=0
& \qquad \text{for } x\in (0,\pi)
\\
 u_{xy}(0,\pm\ell)=u_{xy}(\pi,\pm\ell)=0&
\end{cases}
\end{equation}
with weak form
$$
(u,v)_{H^2_*} =\lambda \int_\Omega p(x,y)\,uv\, dx \, dy\,\qquad\forall v\in H^2\,,
$$
where the scalar product $(\cdot,\cdot)_{H^2_*}$ is defined in Section \ref{problem}. It's worth mentioning that the corner conditions in \eqref{neu} make consistent the two formulations of the problem (classical and weak), while they are unnecessary when dealing with problem \eqref{weight}. Indeed, $u\in H^2_*\cap C^2(\overline \Omega)$ clearly satisfies $u(0,y)=u_y(0,y)=u_{yx}(0,y)=u_{xy}(0,y)=u_{yy}(0,y)=0$ for all $y\in [-\ell, \ell]$ and similarly it happens for $x=\pi$. We refer to \cite[Sections 4]{fergaz} for the derivation of the boundary conditions in \eqref{weight} and to \cite{chas} for those in \eqref{neu}, see also \cite{provenzano} for the formulation of \eqref{neu} in a more general setting.\par
\par \medskip\par
By exploiting the inclusions $H^2_0\subset H^2_*\subset H^2$, we derive
\begin{proposition}
%\label{prop2}
	Let $p\in P_{\alpha,\beta}$ and let $\lambda_h(p)$ be the eigenvalues of \eqref{weight}. Furthermore, denote with $\Lambda_h^{Dir}(p)$ and $\Lambda_h^{Neu}(p)$ the divergent sequences of eigenvalues of \eqref{dir} and \eqref{neu}. There holds:\par
	\begin{equation}\label{comparison1}
	\Lambda_h^{Neu}(p) \leq\lambda_h(p)\leq \Lambda_h^{Dir}(p) \qquad \text{for all } h\in \N_+
	\end{equation} 
	and
	\begin{equation*}\label{asintotic}
	\lambda_h(p)\sim \dfrac{ h^2\, 16\pi^2}{\big(\int_{\Omega}\sqrt{p}\,dxdy\big)^{2}}\qquad as\quad h\rightarrow+\infty.
	\end{equation*}
\end{proposition}
We refer to \cite{buoso} for similar comparisons and a sharper asymptotic analysis in the Neumann case for the homogeneous plate.

\begin{proof}
	The proof is based on the variational characterization of the eigenvalues \eqref{caract1} and on some general results presented in \cite{fleck}. \par
	To prove \eqref{comparison1} we observe that, by density arguments, it follows that $\|\Delta u\|^2_{2}=\|u\|^2_{H^2_*}$ for all $u\in H^2_0(\Omega)$; in this way we may write both the variational representation of Dirichlet and Neumann eigenvalues in \eqref{dir} and \eqref{neu} as
	\begin{equation*}\label{caractnavneu}
	\Lambda_h^{Dir}(p)=\inf_{\substack{W_h\subset H^2_0\\\dim W_h=h}}\hspace{2mm}\sup_{\substack{u\in W_h\setminus \{0\}}}\frac{\|u\|^2_{H^2_*}}{\|\sqrt{p}u\|^2_2}\quad \text{and}\quad\Lambda_h^{Neu}(p)=\inf_{\substack{W_h\subset H^2\\\dim W_h=h}}\hspace{2mm}\sup_{\substack{u\in W_h\setminus \{0\}}}\frac{\|u\|^2_{H^2_*}}{\|\sqrt{p}u\|^2_2}.
	\end{equation*} 
	Observing that $H^2_0\subset H^2_*\subset H^2$, we infer
	\begin{equation*}
	\inf_{\substack{W_h\subset H^2\\\dim W_h=h}}\hspace{2mm}\sup_{\substack{u\in W_h\setminus \{0\}}}\frac{\|u\|^2_{H^2_*}}{\|\sqrt{p}u\|^2_2}\leq\inf_{\substack{W_h\subset H^2_*\\\dim W_h=h}}\hspace{2mm}\sup_{\substack{u\in W_h\setminus \{0\}}}\frac{\|u\|^2_{H^2_*}}{\|\sqrt{p}u\|^2_2}\leq \inf_{\substack{W_h\subset H^2_0\\\dim W_h=h}}\hspace{2mm}\sup_{\substack{u\in W_h\setminus \{0\}}}\frac{\|u\|^2_{H^2_*}}{\|\sqrt{p}u\|^2_2},
	\end{equation*}
	implying inequality \eqref{comparison1}.\par
	Finally, the asymptotic law for $\lambda_h(p)$ follows from \cite[Theorem 4.1, 4.2]{fleck}, where the authors prove 
	$$\Lambda_h^{Dir}(p)\sim16\pi^2\dfrac{ h^2}{\big(\int_{\Omega}\sqrt{p}\,dxdy\big)^{2}}\quad \text{ and }\quad\Lambda_h^{Neu}(p)\sim16\pi^2\dfrac{ h^2}{\big(\int_{\Omega}\sqrt{p}\,dxdy\big)^{2}}\qquad as\quad h\rightarrow+\infty,$$
	implying the same asymptotic behavior for $\lambda_h(p)$.
\end{proof}

The estimate \eqref{comparison1} confirms the general principle that, fixed $h\in\mathbb{N}_+$, any additional constraint increases the eigenvalue and, therefore, the vibration frequency. We point out that the Dirichlet problem represents the most constrained situation, while the Neumann the most free. Problem \eqref{weight} has intermediate boundary conditions, reflecting the trend given by \eqref{comparison1}.

\par\bigskip\noindent
\textbf{Acknowledgments.} The authors are members of the Gruppo Nazionale per l'Analisi Matematica, la Probabilit\`a e le loro Applicazioni (GNAMPA) of the Istituto Nazionale di Alta Matematica (INdAM) and are partially supported by the INDAM-GNAMPA 2019 grant: ``Analisi spettrale per operatori ellittici con condizioni di Steklov o parzialmente incernierate'' and by the PRIN project 201758MTR2: ``Direct and inverse problems for partial differential equations: theoretical aspects and applications'' (Italy).


\begin{thebibliography}{10}
	\bibitem{ammann} O.H. Ammann, T. von K\'{a}rm\'{a}n, G.B. Woodruff, The failure of the Tacoma Narrows Bridge, Federal Works Agency (1941).
	
	\bibitem{anedda1} C. Anedda, F. Cuccu, {\em Steiner symmetry in the minimization of the first eigenvalue in problems involving the p-Laplacian}, Proc. Amer. Math. Soc. 144, (2016), 3431-3440. 
	
	\bibitem{anedda2} C. Anedda, F. Cuccu, G. Porru, {\em Minimization of the first eigenvalue in problems involving the bi-laplacian}, Rev. Mate. Teor. Appl. 16, (2009), 127-136. 
	
	\bibitem{benguria} M.S. Ashbaugh, R.D. Benguria, {\em A sharp bound for the ratio of the first two eigenvalues of Dirichlet Laplacians and extensions},
	Ann. of Math. 135, (1992), 601-–628.
	
	\bibitem{auchmuty}G. Auchmuty, {\em Dual variational principles for eigenvalue problems}, in Nonlinear Functional 	Analysis and Its Applications, F. Browder, ed., American Mathematical Society, Providence, RI,	(1986), 55-71.
	
	%	\bibitem{banerjee} J.R.\ Banerjee, {\em A simplified method for the free vibration and flutter analysis of bridge decks}, J. Sound Vibration 260 (2003), 829-845
	
	%\bibitem{beesack}	P.R. Beesack, {\em Isoperimetric inequalities for the nonhomogeneous clamped rod and plate}, J. Math. Mech. 8 (1959), 471-482.
	
	
	\bibitem{bebuga2} E. Berchio, D. Buoso, F. Gazzola,  {\em On the variation of longitudinal and torsional frequencies in a partially hinged rectangular plate}, ESAIM Control Optim. Calc. Var. 24, (2018), 63-87.
	
	\bibitem{bebugazu}  E. Berchio, D. Buoso, F. Gazzola, D. Zucco, {\em A minimaxmax problem for improving the torsional stability of rectangular plates}, J. Optim. Theory Appl. 177, (2018), 64-92. 
	
	\bibitem{befafega} E. Berchio, A. Falocchi, A. Ferrero, D. Ganguly, {\em On the first frequency of reinforced partially hinged plates}, Commun. Contemp. Math., (2019), 1950074, 37 pp.
	
	\bibitem{bfg}	E.\ Berchio, A.\ Ferrero, F.\ Gazzola, {\it Structural instability of nonlinear plates modelling suspension bridges: mathematical answers to some long-standing questions}, Nonlin. Anal. Real World Appl. 28, (2016), 91-125.
	
	%\bibitem{elvisefilippo} E. Berchio, F. Gazzola, {\em A qualitative explanation of the origin of torsional instability in suspension bridges},	Nonlin. Anal. TMA 121 (2015), 54-72.
	
	\bibitem{bgz}
	E.\ Berchio, F.\ Gazzola, C.\ Zanini, {\it Which residual mode captures the energy of the dominating mode in second order Hamiltonian systems?}, SIAM J. Appl. Dyn. Syst. 15, (2016), 338-355.
	
	
	%\bibitem{bleich} F. Bleich, C.B. McCullough, R. Rosecrans, G.S. Vincent, The mathematical theory of vibration in suspension bridges, U.S. Dept. of Commerce, Bureau of Public Roads, Washington D.C. (1950)
	
	\bibitem{bogamo} D. Bonheure, F. Gazzola, E. Moreira dos Santos, {\em Periodic solutions and torsional instability in a nonlinear nonlocal plate equation}, to appear in SIAM J. Math. Anal. 
	
	\bibitem{buoso} D. Buoso, L. Provenzano, J. Stubbe,  {\em  Semiclassical bounds for spectra of biharmonic operators}, arXiv:1904.11877.
	
	
	\bibitem{chanillo} S. Chanillo, D. Grieser, M. Imai, K. Kurata, I. Ohnishi, {\em Symmetry breaking and other phenomena in the optimization of eigenvalues for composite membranes}, Comm. Math. Phys. 214, (2000), 315-337.
	
	 \bibitem{chas} L.M. Chasman, {\em An isoperimetric inequality for fundamental tones of free plates},  Commun. Math. Phys. 303(2), (2011), 421-449.  
	
	\bibitem{chen} W. Chen, C-S. Chou, C-Y. Kao, \textit{Minimizing Eigenvalues for Inhomogeneous Rods and Plates}, J. Sci. Comput. 69, (2016), 983–1013.
	
	
	\bibitem{CV} F.~Colasuonno, E.~Vecchi, {\em  Symmetry in the composite plate problem}, Commun. Contemp. Math., Vol. 21, No. 02, (2018), 1850019, 34 pp.
	
	\bibitem{CV2} F.~Colasuonno, E.~Vecchi, {\em Symmetry and rigidity for the hinged composite plate problem}, J. Differential Equations 266, (2019), no. 8, 4901--4924.
	
	\bibitem{copro} B. Colbois, L. Provenzano, {\em Eigenvalues of elliptic operators with density,} Calculus of Variations and Partial Differential Equations, (2018), 57-36.
	
		\bibitem{como} M. Como, S. Del Ferraro, A. Grimaldi, A parametric analysis of the flutter instability for long span suspension bridges.  Wind and Structures 8, (2005), 1-12. 
	
		\bibitem{courant} R. Courant, D. Hilbert, Methods of Mathematical Physics, vol. 1. Interscience Publishers Inc, New York, (1953).
	
	\bibitem{cox} S.J. Cox, J.R. McLaughlin, {\em Extremal eigenvalue problems for composite membranes}, I, II, Appl. Math. Optim., 22, no. 2,  (1990), 153-167, 169-187. 
	
	\bibitem{cuccu22} F. Cuccu, G. Porru, {\em Maximization of the first eigenvalue in problems involving the bi-Laplacian}, Nonlinear Anal. 71,  (2009), 800-809.
	
	% \bibitem{delfour} Delfour, M.C., Zolésio, J.-P., Shapes and Geometries, Volume 22 of Advances in Design and Control, 2nd edn. Society for Industrial and Applied Mathematics (SIAM), Philadelphia (2011). Metrics, analysis, differential calculus, and optimization
	
	\bibitem{fergaz} A.\ Ferrero, F.\ Gazzola, {\it A partially hinged rectangular plate as a model for suspension bridges}, Disc.\ Cont.\ Dyn.\ Syst.\ A.\ 35, (2015), 5879-5908.
	
	\bibitem{fleck}J. Fleckinger, M.L. Lapidus, {\it Eigenvalues of elliptic boundary value problems with an indefinite weight function}. Trans. Am. Math. Soc. 295(1), (1986), 305–324. 
	
	\bibitem{bookgaz} F.\ Gazzola, Mathematical models for suspension bridges, MS\&A Vol.\,15, Springer, (2015).
	
	%\bibitem{book} F. Gazzola, H.C. Grunau, G. Sweers, {\em Polyharmonic boundary value problems}, LNM 1991 Springer, (2010).
	
	\bibitem{henrot} A. Henrot, Extremum problems for eigenvalues of elliptic operators, Frontiers in mathematics, Birk\"auser Verlag, Basel, Boston, Berlin,
	(2006).
	
	%\bibitem{herrmann} G. Herrmann and W. Hauger, {\em On the interrelation of divergence, flutter and auto-parametric resonance.} Vol 42 of {\em Ingenieur-Archiv}, (1973), 81-88.
	
	\bibitem{Irvine} H.M. Irvine, Cable structures, MIT Press Series in Structural Mechanics, Massachusetts, (1981).
	
	%\bibitem{jhnm} J.A.\ Jurado, S.\ Hern\'andez, F.\ Nieto, A.\ Mosquera,  Bridge aeroelasticity: sensitivity analysis and optimum design (high performance
	%	structures and materials), WIT Press - Computational Mechanics, Southampton, (2011).
	
	%\bibitem{ksw} Kawohl, B., Star\'a, J., Wittum, G.: Analysis and numerical studies of a problem of shape design. Arch. Rat. Mech. Anal. 114, 349-363 (1991)
	
	\bibitem{keller}J.B. Keller, {\em The minimum ratio of two eigenvalues}, SIAM J. Appl. Math., 31, (1976), no. 3,  485-491.
	
	%\bibitem{Kirchhoff} G.R. Kirchhoff, {\em {\"U}ber das gleichgewicht und die bewegung einer elastischen scheibe}, J. Reine Angew. Math. {\bf 40}, 51-88 (1850)
	
	%\bibitem{lacarbonara} W.\ Lacarbonara, {\em Nonlinear structural mechanics}, Springer (2013)
	
	%\bibitem{lala2004} P.D.\ Lamberti, M.\ Lanza de Cristoforis, {\it A real analyticity result for symmetric functions of the eigenvalues of a domain dependent Dirichlet problem for the Laplace operator}, J.\ Nonlinear Convex Anal., 5 (2004), no.\ 1, 19-42
	
	\bibitem{lapr} P.D.\ Lamberti, L. Provenzano, {\it A maximum principle in spectral optimization problems for elliptic operators subject to mass density perturbations}, Eurasian Math. J. 4, (2013), no. 3, 70-83.
	
	\bibitem{larsen} A. Larsen, {\it Aerodynamics of the Tacoma Narrows Bridge - 60 years later}, Struct. Eng. Internat. 4, (2000), 243-248.
	
	%\bibitem{Love} A.E.H. Love, {\em A treatise on the mathematical theory of elasticity (Fourth edition)}, Cambridge Univ. Press (1927)
	
	%\bibitem{mansfield} E.H. Mansfield, {\em The bending and stretching of plates}, 2nd edition, Cambridge Univ. Press (2005)
	
	%\bibitem{murat} Murat, F., Tartar, L.: Calculus of variations and homogenization. Topics in the Math. Modelling of Composite Materials 31, Progr. Nonlin. Diff. Eq. Appl. 139-173 (1997)
	
	\bibitem{navier} C.L.\ Navier, {\it Extraits des recherches sur la flexion des plans \'elastiques}, Bulletin des Sciences de la Soci\'et\'e
	Philomathique de Paris, (1823), 92-102.
	
	\bibitem{provenzano} L. Provenzano, {\it A note on the Neumann eigenvalues of the biharmonic operator,} Math. Methods Appl. Sci., 41(3), (2018), 1005-1012.
	
	%\bibitem{nazarov} Nazarov, S.A., Sweers,   G.H.,  Slutskij, A.S.: Homogenization of a thin plate reinforced with periodic families of rigid rods. Sbornik Mathematics 202, 1127-1168 (2011)
	
	%\bibitem{oudet} E.\ Oudet, {\it  Numerical minimization of eigenmodes of a membrane with respect to the domain}, ESAIM COCV 10 (2004), 315-335 R.\ Scott, {\em In the wake of Tacoma. Suspension bridges and the quest for aerodynamic stability}, ASCE Press (2001)
	
	\bibitem{rocard} Y. Rocard, Dynamic instability: automobiles, aircraft, suspension bridges. Crosby Loockwood, London, (1957).
	%\bibitem{sha} H. Shahgholian, {\em The singular set for the composite membrane problem}, Comm. Math. Phys. 271, (2007), 93-101
	
	%\bibitem{flutter} YouTube, Airfoil flutter, aircraft flutter, bridge flutter. Available at: \tt https://www.youtube.com/watch?v=72cQgXw7\_kI; https://www.youtube.com/watch?v=iTFZNrTYp3k; https://www.youtube.com/watch?v=1Oq8HuB7\_tI
	
	
\end{thebibliography}
\end{document}